\documentclass[12pt]{amsart}
\usepackage{etex}

\usepackage{amsmath}
\usepackage{mathrsfs, euscript}
\usepackage{bbm}
\usepackage{amssymb}
\usepackage{amscd}
\usepackage[all,cmtip]{xy}
\usepackage{enumerate}
\usepackage{mathabx}

\newtheorem{theorem}{Theorem}[section]
\newtheorem{lemma}[theorem]{Lemma}
\newtheorem{proposition}[theorem]{Proposition}
\newtheorem{corollary}[theorem]{Corollary}
\newtheorem{example}[theorem]{Example}
\newtheorem{remark}[theorem]{Remark}
\newtheorem{definition}[theorem]{Definition}
\newtheorem{notation}[theorem]{Notation}
\newtheorem{proposition-definition}[theorem]{Proposition-Definition}

\newcommand{\C}{\mathbb{C}}

\newcommand{\range}{\mathcal{B}}
\newcommand{\domain}{\mathcal{A}}
\newcommand{\E}{\mathbb{E}}
\newcommand{\M}{\mathcal{M}}
\newcommand{\MC}{\mathcal{M}^\circ}
\newcommand{\I}{\mathcal{I}}
\newcommand{\otimesb}{\otimes_\range}
\newcommand{\hh}{\mathcal{H}}
\newcommand{\K}{\mathcal{K}}
\newcommand{\J}{\mathcal{J}}
\newcommand{\N}{\mathbb{N}}
\newcommand{\ii}{{\bf i}}

\newcommand{\bp}{\boxplus}
\newcommand{\br}{\boxright}
\newcommand{\B}{\mathbb{B}}
\newcommand{\cp}{\mathcal{CP}(\range)}
\newcommand{\mnc}{M_n(\C)}
\newcommand{\cbd} {\Sigma_0(\range)}

\begin{document}
\unitlength=1mm
\special{em:linewidth 0.4pt}
\linethickness{0.4pt}
\title[Convolutions and Transforms]{ Relations between  convolutions and transforms in operator-valued free probability}
\author{Weihua Liu}
\maketitle

\begin{abstract}
We introduce a class of independence relations, which include free, Boolean and monotone independence, in operator valued probability. 
We show that this class of  independence relations have a matricial extension property so that we can easily study their associated convolutions via Voiculescu's fully matricial function theory.  
Based the matricial extension property, we show that  many results can be generalized to multi-variable cases. 
Besides free, Boolean and monotone independence convolutions,  we will focus on two important convolutions, which are orthogonal and subordination additive convolutions.
We show that the operator-valued subordination functions, which come from the  free additive convolutions or the operator-valued free convolution powers, are reciprocal Cauchy transforms of operator-valued random variables which are uniquely determined up to Voiculescu's fully matricial function theory. 
In the end, we  study relations between certain convolutions and transforms in $C^*$-operator valued probability.

\end{abstract}

\section{Introduction}
\subsection{Convolutions and transforms in $\C$-valued probability}
Free probability is a theory introduced by Voiculescu for studying the reduced free products of operator algebras with specified states \cite{VDN,NS}. 
For example, $C^*$-algebras with norm continuous states and Von Neumann algebras with normal states are frameworks  for this theory.   
The elements from the algebras are called noncommutative  random variables or random variables for short. 
A single selfadjoint random variable can be identified with  a classical $\mathbb{R}$-valued random variable thus one can study its distribution. 
The notion of free independence is a key relation for random variables in free probability theory.
Given two freely independent random variables $X$ and $Y$ whose distributions are $\mu$ and $\nu$, respectively, 
then  the distribution of $X+Y$ is denoted by $\mu\bp\nu$, which is called the free additive convolution of $\mu$ and $\nu$. 
In the process of computing the distribution $\mu\bp\nu$,  the Cauchy transform plays  a central role.   Given a probability measure  $\mu$ on $\mathbb{R}$, then the Cauchy transform of $\mu$ is the function
\begin{equation} 
G_\mu(z)=\int_{\mathbb{R}}\frac{1}{z-t}d\mu(t):\C^+\rightarrow \C^-.
\end{equation}
Voiculescu introduced the R-transform of $\mu$  to be 
$$R_\mu(z)=G_\mu^{\langle-1\rangle}(z)-\frac{1}{z},$$
where $G_\mu^{\langle-1\rangle}$ is the inverse of $G_\mu$ with respect to composition. Voiculescu showed that 
\begin{equation}\label{additive r-transform}
R_{\mu\bp\nu}= R_\mu+R_\mu,
\end{equation}
when $\mu$ and $\nu$ are compactly supported probability measures \cite{Voi8,Voi9}.  Later, Bercovici and Voiculescu proved  the relation (\ref{additive r-transform}) to probability measures with non-compact support \cite{BV2}.
Based on the linearization property (\ref{additive r-transform}) of the R-transform, one can define convolution powers with respect to $\bp$ such that if $\mu$ is a probability measure and $t>0$ is a real number, the convolution power $\mu^{\bp t}$( when it exists) is a probability measure uniquely determined by $$R_{\mu^{\bp t}}=tR_{\mu}. $$
It is well known that $\mu^{\bp t}$ defines for all $t\geq 1$ and $\mu$ \cite{NS}.

In $\C$-valued noncommutative probability,  there are  independence relations other than free indenpendence, for example,  Boolean independence \cite{SW}, monotone independence  \cite{Mu2},  s-free independence  \cite{Len3} and orthogonal independence \cite{Len3}.  Similarly, we can define  the Boolean additive convolution $\uplus$, the monotone additive  convolution $\rhd$, the  s-free additive  convolution $\br$,  the orthogonal additive  convolution $\vdash$ for proability measures.  Among these independence relations,  Boolean independence is only commutative non-unital universal  independence relation \cite{Sp1}.  Given a probability measure $\mu$, we denote by $F_\mu$ the reciprocal Cauchy transform of $\mu$, that is 
$$F_{\mu}(z)=\frac{1}{G_\mu(z)}:\C^+\rightarrow \C^+.$$  The Boolean additive convolution has the following property
\begin{equation}\label{Scalar boolean additive}
F_{\mu\uplus \nu}(z)=F_{\mu}(z)+F_{\nu}(z)-z
\end{equation}
where $\mu$,$\nu$ are probability measures and $\mu\uplus \nu$ is the Boolean convolution of $\mu$ and $\nu$ \cite{Franz}. It is similar to the free additive convolutions, we have Boolean convolution powers $\mu^{\uplus t}$ such that if $t>0$ and $\mu$ is a probability measure, then $ \mu^{\uplus^t}$(always exists) is a probability uniquely determined by 
$$F_{\mu^{\uplus t}}(z)=tF_{\mu}(z)+(1-t)z.$$

Monotone independence relation is a noncommutative relation but  defines an associative convolutions.   The monotone additive convolution has the following property
\begin{equation}\label{Scalar monotone additive}
F_{\mu\rhd\nu}(z)=F_{\mu}(F_{\nu}(z))
\end{equation}
where $\mu$,$\nu$ are probability measures and $\mu\rhd\nu$ is the monotone convolution of $\mu$ and $\nu$ \cite{Franz}.   

Orthogonal independence relation is introduced by Lenczewski \cite{Len3} to study decompositions of the free additive convolutions.
The  additive convolution $\vdash$ associated with the orthogonal independence relation is neither  commutative  nor associative.  
The orthogonal additive convolution has the following analytic property
\begin{equation}\label{orthogonal additive}
F_{\mu\vdash\nu}(z)=F_\mu(F_\nu(z))-F_\nu(z)+z
\end{equation}
where $\mu$, $\nu$ are probability measures and $\mu\vdash\nu$ is the orthogonal additive convolution of $\mu$ and $\nu$ \cite{Len3}. From the relations (\ref{Scalar boolean additive}), (\ref{Scalar monotone additive}) and (\ref{orthogonal additive}), one can easily see that 
\begin{equation}
\mu\rhd \nu=\nu\uplus(\mu\vdash \nu).
\end{equation}

The possibility for studying the free additive convolutions via Boolean and monotone additive convolutions is 
 based on the fact that  the Cauchy transform of the distribution  $\mu\bp\nu $ is  subordinated to the Cauchy transforms of $\mu$ and of $\nu$ on  the upper half-plane $\C^+$\cite{Voi6}. 
That is
\begin{equation}\label{Subordination 1}
F_{\mu\bp\nu}(z)=F_\mu(\omega_1(z))= F_\nu(\omega_2(z)).
\end{equation}

It is shown that $ \omega_1(z)$ and $ \omega_2(z)$ are reciprocal Cauchy transforms of certain distributions which are completely determined by $\mu$ and $\nu$ \cite{CF}. 
Therefore, subordination functions provide a new type of convolution, which is called  subordination convolution or s-free additive convolution \cite{Len3}. 
 It is denoted by $\mu \br \nu$ whose reciprocal Cauchy transform is $\omega_1(z)$ in Equation (\ref{Subordination 1}).  In summary, we have the following property
 \begin{equation}\label{Subordination 2}
F_{\mu\bp\nu}(z)=F_\mu(F_{\nu\br\mu}(z))= F_\nu(F_{\mu\br\nu}(z))=F_{\nu\br\mu}(z))+F_{\mu\br\nu}(z)-z
\end{equation}
or
 \begin{equation}\label{Subordination 2}
\mu\bp\nu=\mu\rhd({\nu\br\mu})= \nu\rhd({\mu\br\nu})=(\nu\br\mu)\uplus (\mu\br\nu).
\end{equation}
On should be careful that the s-free additive convolution $\br$ is neither commutative nor associative. However,  $\br$ is very powerful in studying free additive convolutions \cite{ BB, Bel1, Bel2}. 

Let us brief review some relations between the convolutions we introduced before.  
In \cite{BN3}, Belinschi and Nica introduced a family $\{\B_t|t>0\}$ of transforms on probability measures by the formula
$$\B_t(\mu)=(\mu^{\bp (1+t)})^{\uplus\frac{1}{1+t}}.$$
When $t=1$, the transform is a remarkable bijection discovered by Bercovici and Pata in their study of relations between $\bp$-infinite divisibility and $\uplus$-infinite divisibility\cite{BP}.
Belinschi and Nica showed that these transforms form a semigroup with respect to composition, that is 
 \begin{equation}\label{scalar B-P bijection}
 \B_t(\B_s(\mu))=\B_{s+t}(\mu), \forall s,t>0.
 \end{equation}
Then, they showed that the relation (\ref{scalar B-P bijection}) can be derived by the following relation between the free convolution powers and the  Boolean convolution powers
\begin{equation}\label{scalar FB-BF}
(\mu^{\bp p})^{\uplus q}=(\mu^{\uplus q'})^{\bp p'},
\end{equation}
where $p\geq 1$, $q>\frac{p-1}{p}$, $p'=\frac{pq}{1-p+pq}$, $q'=1-p+pq$ and $\mu$ is an algebraic probability distribution \cite{BN1}.

Let $\gamma$ be the free central limit law with variance $1$ in free probability.  In \cite{BN1},  Belinschi and Nica studied the free Brownian motion started at an arbitrary algebraic probability measure $\mu$ in the process $\{\mu\bp \gamma^{\bp t}|t\geq 0\}$. 
They defined a transform $\Phi$ for probability measures such that given a probability measure $\mu$ on $\mathbb{R}$, then $\Phi(\mu)$ is a probability measure on $\mathbb{R}$ such that 
\begin{equation}\label{Phi-transform}
F_{\Phi(\mu)}(z)=z-G_{\mu}(z).
\end{equation}
Belinschi and Nica discovered that
\begin{equation}\label{Phi and BP}
\Phi(\mu\bp \gamma^{\bp t} )=\B_{t}(\Phi).
\end{equation}
Later,  Anshelevich generalized the relation (\ref{Phi and BP}) further by considering the free process $ \{\mu\bp \nu^{\bp t}|t\geq 1\}$  where $\nu$ is an arbitrary probability measure \cite{An}. 

Almost the above relations are interpreted via s-free additive  convolutions which is fully developed in a combinatorial way by Nica \cite{Nica}.
First, Nica exhibited the following relation between the free convolutions and  the s-free additive convolutions
\begin{equation} \label{scalar free addtive and subordination}
(\mu_1\bp\mu_2)\br \nu=(\mu_1\br \nu) \bp (\mu_2\br \mu )
\end{equation}
and 
\begin{equation}\label{scalar free powers and subordination}
(\mu^{\bp t})\bp \nu=(\mu\br\nu)^{\bp t}
\end{equation}
where $\mu_1$, $\mu_2$, $\mu$ and $\nu$ are probability measures and $t$ is a positive real number. Nica also found the following relations between subordination convolutions and $\B$-transform, $\Phi$-transform
\begin{equation}\label{scalar subordination and B}
\mu\br\mu=\B_1(\mu)
\end{equation}
and 
\begin{equation}\label{semicircular Phi B}
\gamma\br\mu=\B_1(\Phi(\mu))
\end{equation}
and 
\begin{equation}
\mu^{\bp p}=\mu\rhd\left(\B_1(\mu))^{\bp p-1}\right).
\end{equation}
Then, Anshelevich's free evolution with two states can be defined by  the following formula
\begin{equation}
\Phi[\mu,\nu]=\B_1^{-1}(\mu\br \nu)
\end{equation}
and satisfies the following relation
\begin{equation}
\Phi[\mu,\nu\bp \mu^{\bp t}]=\B_t(\Phi[\mu,\nu])
\end{equation}
or equivalently
\begin{equation}\label{Anshelevichi transform}
\B_t(\mu\br \nu)=\mu\br(\mu^{\bp t}\bp \nu),
\end{equation}
where $\mu,\nu$ are probability measures and $t\geq 1$.

It should be pointed out here that the results of Anshelevich, Belinschi and Nica are more general than what are listed above. In short, they considered $k$-tuples of random variables.  In this paper, we are gonna show that $k$-tuples of random variables are special cases of operator valued random variables.

\subsection{Operator valued probability}
Soon after the birth of free probability, Voiculescu generalized the theory  to an operator valued framework where the states in the original theory are replaced by bimodule maps \cite{Voi5}.  
The initial analytic type of $R-$transform in operator valued free probability  was first introduced by Voiculescu \cite{Voi5} and   the combinatorial aspects of operator valued R-transforms were developed by Speicher \cite{Speicher}. 
Similarly, the  operator valued R-transforms have the additive property (\ref{additive r-transform}). 
It follows that the convolution powers with respect to $\bp $ can be naturally generalized in operator-valued frame work if one uses positive real number power. 
Following a question due to Bercovici, in $C^*$-operator valued probability, the convolution powers with respect to $\bp $ was generalized further by Anshelevich, Belinschi, Fevrier, and Nica  such that $t$  can be  a completely positive map \cite{ABFN}.  
An explicit construction of the operator-valued free convolution powers was given by Shlyakhtenko \cite{Sh1}.  Furthermore,  a combinatrorial definition for the Belinschi-Nica's$\Phi$ is provided,  then  the relation (\ref{scalar B-P bijection}) and the relation (\ref{Phi and BP}) are generalized in  operator valued probability  \cite{ABFN}.  

On the other hand,  operator valued subordination functions  were discovered Biane \cite{Biane} and were reformulated by Voiculescu \cite{Voi4,Voi10} . 
As well as the scalar subordination functions, the subordination property is a very useful analytic property in operator-valued probability \cite{BMS,BSTV,BPV}. 
By studying the subordination property further,  Voiculescu developed his fully matricial  function theory \cite{Voi3, Voi2}, from which the main tool in this paper is derived, for studying operator valued probability. A direct generalization of Cauchy transform  is  usually not enough to determine an operator-valued distribution, but this can be done by using Voiculescu's fully matricial operator valued functions.

The goals of this paper are the following:  In $C^*$-operator valued probability, we show that  operator-valued subordination functions completely determine  operator valued distributions in the fully matricial operator function theory.  
Therefore, we can define operator-valued subordination convolutions or s-free additive convolutions.  
Furthermore,  we show that Belinschi-Nica' s $\Phi$-transform is well defined in $C^*$-operator valued probability by using a fully matricial version of the equation (\ref{Phi-transform}).  
Then,  we prove relations (\ref{scalar FB-BF})-(\ref{Anshelevichi transform}) in operator valued cases.
We also show that many properties related to $k$-tuples of  classical  random variables can be reduced properties related to operator valued random variables.

Besides this introduction section, the rest of this paper is organized as follows:  
In Section 2, we  introduce the main notions and  tools to be used in this paper. 
In Section 3, we introduce a class of independence relations which are called $\J$-independence  and their associated convolutions. 
We show a matricial property of the $\J$ independence relation, and explain that  the $\J$-convolutions of $k$-tuples of $\range$-valued random variables can be decoded from  the$\J$-convolutions of $\range\otimes \mnc$-valued random variables.
In Section 4, we study the reciprocal Cauchy transforms of Boolean convolutions.
In Section 5, we study the reciprocal Cauchy transforms of Monotone convolutions.
In Section 6, we introduce and study the  operator valued orthogonal additive convolutions.
In Section 7, we introduce and study the  operator valued s-free additive convolutions.
After the preparations in the previous sections, in Section 8, we generalize the relations in Section 1.1 to the operator valued case.

\section{Preliminary}

\subsection{$\range$-valued Independence relations }  We start with some necessary definitions in operator probability.

\begin{definition} \normalfont A $\range$-valued probability space $(\domain, \E:\domain\rightarrow \range)$ consists of a unital algebra  $\range$, a unital algebra $\domain$ which is also a $\range$-$\range$ bimodule and a conditional expectation $\E:\domain\rightarrow \range$ i.e.
 $$\E[b_1ab_2]=b_1\E[a]b_2,  \,\,\, \E[b 1_\domain]=b,$$
for all $b_1,b_2,b\in\range$, $a\in\domain$ and $1_\domain$ is the unit of $\domain$.  The elements of $\domain$ are called $\range$-valued random variables.
Suppose that  $\range$ is a unital $C^*$-algebra and $\domain$ is a $*$-algebra,   the conditional expectation $\E$ is said to be  positive if 
 $$\E[aa^*]\geq 0,$$
 for  all $a\in \domain.$ An element $x\in \domain$ is selfadjoint if $x=x^*$.
 
\end{definition} 

We will denote by $(\domain,\E)$ short for $(\domain, \E:\domain\rightarrow \range)$ when there is no confusion. Throughout this paper,  we will assume that $\domain$ and $\range$   are   unital  $C^*$-algebras.
In this case, $(\domain,\E)$ is a called a $C^*$-operator valued probability space.

\begin{definition}\normalfont We denote by $\range\langle X\rangle$  the algebra which is freely generated by $\range$ and the indeterminant $X$. The elements in  $\range\langle X\rangle$ are called $\range$-polynomials. In addition,  $\range\langle X\rangle_0$ denotes the subalgebra of $\range\langle X\rangle$ whose elements do not contain a constant term, i.e. the linear span of the noncommutative monomials $b_0Xb_1Xb_2\cdots b_{n-1}Xb_n$, where $b_k\in\range$ and $n\geq 1$. 
\end{definition}

\begin{definition}\normalfont
Given a a $\range$-valued probability space $(\domain,\E)$ and an element $a\in \domain$, the \emph{distribution} $\mu_a$ of $a$ is is a map from $\range\langle X\rangle$ to $\range$ such that 
$$ \mu_a(P)=\E[P(a)], $$
for all $P\in \range\langle X\rangle.$ A \emph{$\range$-valued distribution} $\mu$ is  is a map from $\range\langle X\rangle$ to $\range$ such that $\mu=\mu_a$ from some $\range$-valued random variable $a$. If $a$ can be chose from $C^*$-operator $\range$-valued probability space, then $\mu$ is called a $\range$-valued distribution and will be denoted by $\mu\in \cbd$.
\end{definition}

\begin{definition}\normalfont
Given a $\range$-valued probability space $(\domain,\E)$.
\begin{itemize}
\item A family of unital $\range$-$\range$-bimodule subalgebras  $\{\domain_i\}_{i\in I}$ of $\domain$  are said to be freely independent with respect to $\E$ if 
$$\E[a_1\cdots a_n]=0,$$
whenever $i_1\neq i_2\neq \cdots\neq i_n$, $a_k\in \domain_{i_k}$ and $\E[a_k]=0$ for all $k$. 
A family of $ (x_i)_{i\in I}$ are said to be freely independent over $\range$, if the unital $\range$-$\range$-bimodule subalgebras $\{\domain_i\}_{i\in I}$ which are generated by $x_i$, respectively,  are freely independent, or equivalently 
$$\E[p_1(x_{i_1})p_2(x_{i_2})\cdots p_n(x_{i_n})]=0,$$
whenever $i_1\neq i_2\neq \cdots\neq i_n$, $p_1,...,p_n\in \range\langle X\rangle$ and $\E[p_k(x_{i_k})]=0$ for all $k$.

\item A family of non-unital $\range$-$\range$-bimodule subalgebras $\{\domain_i \}_{i\in I}$ of $\domain$  are said to be boolean independent with respect to $\E$ if 
$$\E[a_1\cdots a_n]=\E[a_1]\E[a_2]\cdots \E[a_n],$$
whenever  $a_k\in \domain_{i_k}$ and $i_1\neq i_2\neq\cdots\neq i_n$.   
 A family of random variables $\{x_i\}_{i\in I}$ are said to be boolean independent over $\range$, if the non-unital $\range$-$\range$-bimodule subalgebras $\{\domain_i\}_{i\in I}$ which are generated by $x_i$  respectively are boolean independent, or equivalently 
$$\E[p_1(x_{i_1})p_2(x_{i_2})\cdots p_n(x_{i_n})]=\E[p_1(x_{i_1})]\E[p_2(x_{i_2})]\cdots \E[p_n(x_{i_n})],$$
whenever $i_1\neq i_2\neq\cdots\neq i_n$ and $p_1,...,p_n\in\range\langle X\rangle_0$.

\item If $\I$ is ordered,with order $>$. A family of non-unital $\range$-$\range$-bimodule subalgebras $\{\domain_i \}_{i\in I}$  of $\domain$   are said to be monotone independent with respect to $E$ if 
$$\E[a_1\cdots a_{k-1}a_ka_{k+1} a_n]=\E[[a_1\cdots a_{k-1}\E[a_k]a_{k+1} a_n],$$
whenever  $a_k\in \domain_{i_k}$ and $i_k>i_{k-1},i_{k+1}$.    A family of random variables $\{x_i\}_{i\in I}$ are said to be boolean independent over $\range$, if the non-unital $\range$-$\range$-bimodule subalgebras $\{\domain_i\}_{i\in I}$ which are generated by $x_i$  respectively are monotone independent.
\end{itemize}
\end{definition}

\begin{definition}\normalfont
Let $\mu,\nu\in\cbd$ such that $\mu$ is equal to the distribution of $x$ and $\nu$ is the distribution of $y$, where $x$ and $y$ are  freely(Boolean, Monotone) independent in a $\range$-valued probability space $(\domain,\E)$.  Then the distribution of $x+y$ is called the free(Boolean, Monotone)  additive convolution of $\mu$ and $\nu$, and will denoted by $\mu\boxplus\nu$.($\mu\uplus\nu$,$\mu\rhd\nu$)
\end{definition}

\subsection{Hilbert $C^*$-modules}  
Our  constructions of $\range$-valued faimily of random variables of independence relations  rely  on the theory of Hilbert $C^*$-modules. We refer  to \cite{Lance} for a good introduction to this theory. If $\range=\C$, the theory becomes the classical Hilbert space theory.  

\begin{definition}\normalfont
 An inner-product $\range$-module is a  linear space $\M$ which is a $\range$-right module together with a map $(x,y)\rightarrow \langle x,y\rangle_{\M}:\M\otimes \M\rightarrow \range$ such that 
 \begin{itemize}
 \item[1.] $\langle x+y,z \rangle_{\M}=\langle x,z  \rangle_{\M}+\langle y,z \rangle_{\M}$
 \item[2.]$\langle x,yb \rangle_{\M}=(\langle x,y\rangle_{\M}) b$
 \item[3.]$\langle x,y\rangle_{\M}=\langle y,x \rangle_{\M}^*$
 \item[4.]$\langle x,x \rangle_{\M}> 0$ if $x\neq 0$
 \end{itemize}
 for all $x,y,z\in \M$ and $b\in \range$.\\
 $\M$ is a $\range$-$\range$ Hilbert bimodule if $\M$ is a $\range$-$\range$  bimodule and is complete with respect to the inner product  $\langle ,\rangle_{\M}.$
\end{definition}

\begin{definition}\normalfont Let $\M$ be a $\range$-$\range$ Hilbert bimodule. A map $t:\M\rightarrow \M$ is adjointable if there is a map $t^*:\M\rightarrow \M$ such that 
$$ \langle tx,y\rangle_{\M}=\langle x,t^*y\rangle_{\M}$$
for all $x,y\in M$. The set of all adjointable maps from $\M_i$ to $\M_i$ is denoted by $L(\M)$.
\end{definition}
It is well known that $L(\M)$ is a $C^*$-algebra.
\begin{definition}\normalfont
A  $\range$-$\range$ Hilbert bimodule with a specified vector is a pair $(\M,\xi)$ consists of  a $\range$-$\range$ Hilbert bimodule $\M$ and a vector $\xi\in \M$ such that  $\langle \xi,b\xi \rangle_{\M}=b$ for $b\in \range.$  
$\MC$ is the orthogonal completion of $\range\xi$.  Then $\M=\range\xi\oplus \MC$, $(L(M),\langle \xi,\cdot\xi \rangle_{\M})$ is a $C^*$-operator $\range$-valued probability space. 
\end{definition}

\begin{notation}\normalfont
  We will denote  $\langle \xi,\cdot\xi \rangle_{\M}$ by $\phi_\xi(\cdot)$.
\end{notation}

\begin{lemma}Let $(\M,\xi)$ be a  $\range$-$\range$ Hilbert bimodule with a specified vector. Then 
$$b\xi=\xi b,$$ for all $b\in\range.$ 
\end{lemma}
\begin{proof}
Since  $\xi b\in \range \xi$, we have 
$$ \langle b'\xi+m,b\xi\rangle_\M=b'^*b=\langle b'\xi+m,\xi\rangle_\M b=\langle b'\xi+m,\xi b\rangle_\M$$
for all $b'\in \range$ and $m\in\MC$.
The proof is done.
\end{proof}

Let $x\in L(\M)$ be a selfadjoint operator. Then   we can write 
$$x=\left(\begin{array}{cc}
p&a\\
a^*&T
\end{array}\right),$$
where $a:\range\xi\rightarrow \MC$, $T:\MC\rightarrow \MC$, $p:\range\xi\rightarrow \range\xi$. Notice that, $p$ can be identified as a map from $\range$ to $\range$, we have the following result.

\begin{lemma}\label{constant corner}
$p$ is a $\range$ constant, namely $p\in \range.$
\end{lemma}
\begin{proof}
$$pb=\langle \xi, pb\xi\rangle_\M=\langle \xi, p\xi b\rangle_\M=\langle \xi, p\xi\rangle_\M b, $$
for all $b\in \range.$
\end{proof}

\begin{notation}\normalfont
We will denote the set $\{b\in \range|\frac{b-b^*}{i}>\epsilon 1_\range,\,\,\text{for some}\,\, \epsilon>0\}$  by $\mathbb{H}^+$ and $\mathbb{H}^-=-\mathbb{H}^+.$
\end{notation}

\begin{definition}\normalfont
Given a distribution $\mu\in\cbd$ such that $\mu=\mu_a$, where $a$ is $\range$-valued random variable in $(\domain,\E)$. The $\range$-Cauchy transform $G_{\mu,1}$ of $\mu$ is a map from $ \mathbb{H}^+\rightarrow \mathbb{H}^-$ defined by
$$ G_{\mu,1}(b)=\E[(b1_\domain-a)^{-1}],$$
whereas $F_{\mu,1}(b)=(G_{\mu,1}(b))^{-1}: \mathbb{H}^+\rightarrow \mathbb{H}^+$ is called the $\range$-reciprocal Cauchy transform of $\mu.$
\end{definition}

Given a distribution $\mu\in\cbd$, following the construction of Section 1.1 in \cite{Dykema}, there exists  a  $\range$-$\range$ Hilbert bimodule with a specified vector $(\M,\xi)$ and a element $x\in L(\M)$ such that $\mu$ is equal to the distribution $\mu_x$ of $x$ in $(L(\M),\phi_\xi).$  

The following proposition shows a relation between the reciprocal Cauchy transform and  the vector-conditional expectation $\phi_\xi$.

\begin{proposition}\normalfont\label{reciprocal Cauchy and representation} Let $b\in \mathbb{H}^+$ and  $x$ be a random variable in $(L(\M),\phi_\xi)$ such that $x$ has the following decomposition
$$x=\left(\begin{array}{cc}
p&a\\
a^*&T
\end{array}\right).$$
Then
    $$F_{\mu_x,1}(b)=b-p-a^*[(b-T)^{-1}|_{\MC}]a,$$
where $[(b-T)^{-1}|_{\MC}]$ is the inverse of the restriction of  $T-b$ to $\MC$.
\end{proposition}
\begin{proof}
Notice that $p$ and $T$ are selfadjoint, hence $b-p$ is invertible in $\range$  and $b-T$ is invertible in $L(\MC)$. With a direct computation, we have that 
$$\left(\begin{array}{cc}
b-p&a\\
a^*&b-T
\end{array}\right)^{-1}=\left(\begin{array}{cc}
(b-p-a^*[(b-T)^{-1}|_{\MC}]a)^{-1} & \cdots\\
\cdots &\cdots\\
\end{array}\right). $$

By Lemma \ref{constant corner},  we have that $(b-p-a^*[(b-T)^{-1}|_{\MC}]a)^{-1}\in \range$.
 Therefore, we have
 \begin{align*}
 F_{\mu_x,1}(b)&= \phi_\xi((b-x)^{-1})^{-1}\\
 &= ((b-p-a^*[(b-T)^{-1}|_{\MC}]a)^{-1})^{-1}\\&=b-p-a^*[(b-T)^{-1}|_{\MC}]a.
 \end{align*}
The proof is done.
\end{proof}

\subsection{Fully matricial functions}  In this subsection, we introduce Voiculescu's fully matricial function theoery in $C^*$-algebra operator valued probility. We also refer to \cite{KV} for a good  development of the theory.   

Let $(\domain,\E)$ be a $C^*$-algebra operator valued probability space.  A fully noncommutative $\domain$-set $\Omega=(\Omega_n)_{n\geq 1}$ is a sequence  such that   $\Omega_n\in   \domain\otimes \mnc$ and $\Omega_{m+n}\cap  (\domain\otimes\mnc\oplus  \domain\otimes M_m(\C))=\Omega_n\oplus \Omega_m$ and $ (1_\domain\otimes S)\Omega_n ( 1_{\domain}\otimes S^{-1})=\Omega_n$ for all $n$, $S\in GL_n(\C)$. 
\begin{definition}\normalfont
 A fully matrial $\range$-valued function on a fully matricial $\domain$-set $\Omega$ is a sequence $(f_n)_{n\geq 1}$ such that
 \begin{itemize}
 \item[1.] $f_n: \Omega_n\rightarrow \range\otimes\mnc$ is a function for each $n$.
 \item[2.] If $g'\in \Omega_m$ and $g''\in\Omega_n$, then $f_{m+n}(g'\oplus g'')=f_n(g')\oplus f_m(g'')$,
 \item[3.] If $S\in GL_n(\C)$ and $g\in \Omega_n$, then $f_n \left((S\otimes 1_\domain)g (S^{-1}\otimes 1_{\domain})\right)= (S\otimes 1_\domain)f_n(g) (S^{-1}\otimes 1_{\domain})$.
 \end{itemize}
\end{definition}
\begin{example}\normalfont For each $n\geq 1$, 
let $(\domain\otimes \mnc,\E\otimes I_n)$ be a $\range\otimes\mnc$-valued probability space such that 
$$(\E\otimes I_n)[(a_{i,j})_{i,j=1,...,n}]=(\E[a_{i,j}])_{i,j=1,...,n},$$
and 
$$\mathbb{H}_n^+=\{b_n\in M_n(\range)|\frac{b_n-b_n^*}{i}>\epsilon 1_{n},\,\,\text{for some}\,\, \epsilon>0\}$$
and $a\in\domain$ be a selfadjoint random variable.
Let $\Omega_n=\{ (a\otimes\I_n-b_n)^{-1}|b_n\in \mathbb{H}_n^+ \}$. Then $\Omega=(\Omega)_{n\geq 1}$ is a fully noncommutative $\domain$-set and $(\E\otimes I_n)_{n\geq 1}$ is a fully matricial function on $\Omega$.
\end{example}
With the help of the  fully matricial function $(\E\otimes I_n)_{n\geq 1}$, we can define the following sequences of functions.
Given a selfadjoint random variable $x\in \domain$, let $\mu$ be the distribution of $a$.  
The matricial Cauchy transform of $\mu$ is a sequence of function  $G_{\mu}=(G_{\mu,n})_{n\geq 1}$ such that 
$$G_{\mu,n}(b_n)= \E\otimes I_n[(b_n-a\otimes I_n )^{-1}]: \mathbb{H}_n^+\rightarrow \mathbb{H}_n^-.$$
In addition, the sequence of functions  $F_\mu=(F_{\mu,n})_{n\geq 1}$, where 
$$F_{x,n}(b_n)=(G_{x,n}(b_n))^{-1}:\mathbb{H}_n^+\rightarrow -\mathbb{H}_n^+$$
is called the  matricial reciprocal Cauchy transform of $\mu$.   We will be using the following fact that illustrated in \cite{BPV}.
\begin{proposition}\normalfont
Let $\mu,\nu\in \Sigma_0(\range)$. Then $\mu=\nu$ if and only if $F_\mu=F_\nu$ on $(\mathbb{H}^+_n)_{n\geq 1}$.
\end{proposition}

The matricial R-transform  $R=(R_{\mu,n})_{n\geq 1}$ of $\mu\in \Sigma_0(\range)$ is a sequence of functions defined as follows:
$$R_{\mu,n}(b_n)=G_{\mu,n}^{\langle-1\rangle}(b_n)-b^{-1},$$
where $G_{\mu,n}^{\langle-1\rangle}$ is the left inverse of $G_{\mu,n}$, namely $G_{\mu,n}^{\langle-1\rangle}(G_{\mu,n}(b_n))=b_n$ for all $b_n\in \mathbb{H}^+_n$. In general,  $R_{\mu,n}$ is not defined on $ \mathbb{H}^+_n $, but in a uniform neighborhood of $0$.  The  matricial Voiculescu transform $\phi_{\mu}=(\phi_{\mu,n})_{n\geq 1}$ is defined as $\phi_{\mu,n}(b)=R_{\mu,n}(b^{-1})$. Then, we have 
$$\phi_{\mu,n}(b_n)=F_{\mu,n}^{\langle-1\rangle}(b_n)-b_n,$$
where $F_{\mu,n}^{\langle-1\rangle}$ is the left inverse of $F_{\mu,n}$.
As the linearization property of the classical R-transform, we have
$$R_{\mu\bp\nu,n}=R_{\mu,n}+R_{\nu,n}$$
and 
$$\phi_{\mu\bp\nu,n}=\phi_{\mu,n}+\phi_{\nu,n},$$
for all $n\geq 1.$

\section{Operator valued independence relations and their matricial extensions}
In this section, we introduce a class of independence relations include free, boolean and monotone independence.  
\begin{definition}\normalfont
Let $\I$ be an index set. 
 We call $$\K(\I)=\{\emptyset\}\amalg\{(i_1,\cdots,i_n)\in\I^n|i_k\neq i_{k+1}\,\, \text{for}\,\,k=1,\cdots,n-1, n\in \N \}$$
 the set of series associated with $\I$.  
\end{definition}
For each $i\in \I$, let  $(\M_i,\xi_i) $ be a $\range$-$\range$ bimodule with a specified vector $\xi_i$ such that  $\langle b_1\xi_i,b_2\xi_i\rangle_{\M_i}=b_1^*b_2$ for all $b_1,b_2\in \range$, where $ \langle \cdot,\cdot\rangle_{\M_i}$ is the inner product on $\M_i$. The reduced free product of $(\M_i,\xi_i)_{i\in\I} $  with amalgamation over $\range$ is  given by
$$\M= \range \xi\oplus \bigoplus\limits_{n=1}^\infty\bigoplus\limits_{i_1\neq i_2\neq \cdots \neq i_n} \MC_{i_1}\otimesb\cdots \otimesb\MC_{i_n}.$$
where the inner product $\langle \cdot, \cdot \rangle_\M$  on $\M$ is recursively defined as follows:
$$\langle m_1\otimes\cdots \otimes m_p, m_1'\otimes\cdots \otimes m_q' \rangle_\M=\delta_{p,q}\delta_{i_1,j_1} \langle m_2\otimes\cdots \otimes m_p, \langle m_1,m_1' \rangle_{\M_{i_1}} m_2'\otimes\cdots \otimes m_q' \rangle_{\M},$$
where $m_k\in \MC_{i_k}$ for $k=1,\cdots p$ and $m_k'\in\MC_{j_k}$ for $k=1,...q.$
$(\M,\xi) $ is a $\range$-$\range$ bimodule with the specified vector $\xi$ such that $\langle b_1\xi,b_2\xi\rangle_{\M}=b_1^*b_2$
\begin{definition}\normalfont
Given $\ii=(i_1,\cdots,i_m)\in \K(\I)$, the $\ii$-projection $P^{\M}_{\ii}$ is the orthogonal projection from $\M$ onto its subspace 
$$ \MC_{i_1}\otimesb \cdots \otimesb\MC_{i_m}.$$
For convenience, $\ii=\{\emptyset\}$ when $m=0$, and $P^{\M}_{\emptyset}$ is the orthogonal projection onto the subspace $ \range \xi$.\\
 Let  $\J\subset \K(\I)$ be  a subset of $\K(\I)$.  We denote by 
 $$P_{\J}=\sum\limits_{\ii\in \J}P^{\M}_\ii.$$ 
\end{definition}

For each $i\in \I$, let 
$$ \MC(i)=  \bigoplus\limits_{n=1}^\infty\bigoplus\limits_{i\neq i_1\neq i_2\neq \cdots \neq i_n} \MC_{i_1}\otimesb\cdots \otimesb\MC_{i_n},$$
and 
$$\M(i)=\range \xi \oplus \MC(i).$$
Then, for each $i$, let $V_i: \M_i\otimesb \M(i)\rightarrow \M$ be the unitary operator defined  as 
$$\begin{array}{cl}
V_i:& \xi_i\otimesb\xi\rightarrow \xi\\
V_i:&\xi_i\otimesb m(i)\rightarrow m(i)\\
V_i:& m_i\otimesb \xi\rightarrow m_i\\
V_i:& m_i\otimesb m(i) \rightarrow  m_i\otimesb m(i),
\end{array} $$ 
where $m(i)\in \MC(i)$, $m_i\in \MC_i$.\\
Let $\lambda_i:L(\M_i)\rightarrow  L(\M)$ be the $*$-homomorphism given by
$$\lambda_i(a)=V_i (a\otimes 1_{\M(i)})V_i ^{-1},$$
where $1_{\M(i)}$ is the identity map on $\M$. It is well-know that the family of $\{\lambda_i(L(\M_i))\}_{i\in I}$ is freely independent over $\range$ with respect to $\phi_{\xi}(\cdot)=\langle \xi,\cdot\xi\rangle_{\M}$.

\begin{definition}\normalfont
Let $\J=\{\J_i\subset \K(\I)\}$ be a family of subsets of $\K(\I)$. We say that $\J$ is $\I$-compatible if  for each $i\in \I$ we have that $(i\neq i_1,\cdots,i_m)\in \J_i$ if and only if 
$(i, i_1,\cdots,i_m)\in J_i.$
 \end{definition}

\begin{proposition}\normalfont
Let $\J=\{\J_i\subset \K(\I)\}$ be a  $\I$-compatible  family of subsets of $\K(\I)$.  Then, for all $i\in\I$,  $P^{\M}_{\J_i}$ is commuting with all elements in $\lambda_i(L(\M_i))$.
\end{proposition}
\begin{proof}
Since the family $\J=\{\J_i\subset \K(\I)\}$ is   $\I$-compatible, for each $i\in \I$, we have that $P^\M_{\J_i}$ is the orthogonal projection onto the subspaces
$$  \bigoplus\limits_{(i\neq i_1,i_2,\cdots, i_m)\in \J_{i}} \MC_{i_1}\otimesb\cdots \otimesb\MC_{i_n}\oplus \bigoplus\limits_{(i\neq i_1,i_2,\cdots, i_m)\in \J_{i}} \MC_i\otimes\MC_{i_1}\otimesb\cdots \otimesb\MC_{i_n}.$$
Let $\J_i'=\{( i_1,i_2,\cdots, i_m)\in \J_{i}|i_1\neq i, m\in \N\}$. Then we have that 
\begin{align*}
\lambda_i(a)P^\M_{\J_i}&= V_i (a\otimes 1_{\M(i)})V_i ^{-1}P^\M_{\J_i}\\
&= V_i (a\otimes 1_{\M(i)})(1_{\M_i}\otimes P^\M_{\J_i'})V_i ^{-1}\\
&= V_i (1_{\M_i}\otimes P^\M_{\J_i'})(a\otimes 1_{\M(i)})V_i ^{-1}\\
&=  P^\M_{\J_i}V_i(a\otimes 1_{\M(i)})V_i ^{-1}.
\end{align*}
The proof is done.
\end{proof}
By the commutativity of the projections defined by an $\I$-compatible family $\J$, we have the following result.
\begin{corollary}\normalfont
Let $\J=\{\J_i\subset \K(\I)\}$ be an $\I$-compatible  family of subsets of $\K(\I)$.  Then, for each $i\in I$, $P_{\J_i}\lambda_i(\cdot)$ is a homomorphism from $L(\M_i)$ into $L(\M)$.
\end{corollary}

Therefore, the following definition is well defined.
\begin{definition}\normalfont
Let $\{\domain_i|i\in \I\}$ be a family of (necessarily unital) subalgebras of a $\range$-valued probability space $(\domain,\E)$. 
 Given an $\I$-compatible family $\J=\{\J_i\subseteq \K\}$ of subsets of $\K$. 
  We say that $\{\domain_i|i\in \I\}$ are $\range$-valued $\J$-independent if there is a family $(\M_i,\xi_i)$ of $\range$-$\range$ bimodules with specified vectors and $\range$-linear homomorphisms $\gamma_i:\domain_i\rightarrow L(\M_i)$, on the reduced free product $(\M,\xi)=*_{i\in I}(\M_i,\xi_i)$, we have 
$$ \E(a_1a_2\cdots a_m)=\phi_\xi(P^\M_{\J_{i_1}}\lambda_{i_1}(\gamma_{i_1}(a_1))\cdots P^\M_{\J_{i_m}}\lambda_{i_m}(\gamma_{i_m}(a_m))) $$
whenever $a_k\in \domain_{i_k}$ for $1\leq k\leq m$.
\end{definition}

\begin{example}\label{Example of freeness}\normalfont For  each $i\in \I$, let $\J_i=\K(\I)$. In this case, the   $\J$-independence is the free independence. 
\end{example}

\begin{example}\label{Example of boolean}
 \normalfont  For  each $i\in \I$, let $\J_i=\{\emptyset, (i)\}$. In this case, the  $\J$-independence is the Boolean independence.
\end{example}

\begin{example}\label{Example of monotone}
 \normalfont Let $\I=\{1,2\}$ with the natural order,  $J_1=\{\emptyset, (1)\}$ and $\J_2=\{\emptyset, (1),(2),(2,1)\}$. 
In this case, the $\J$-independence is the monotone independence.
\end{example}

Now, we show that the $\J$-independence relations have a matricial extension property on which allows us to apply the  matricial functions.

Again, let $(\domain,\E)$ be a $\range$-valued probability space. Let $\mnc$ be the algebra of $n\times n$ matrices. For $1\leq k,l\leq n$,  let $e(k,l)$ be the element whose $(i,j)$-th entry is $1$ and the other entries are $0$.    Let $\E_n=\E\otimes I_n$ be the  map from $\domain\otimes \mnc$  to $\range\otimes \mnc$ defined as 
$$\E_n[a\otimes e(k,l)]=\E[a]\otimes e(k,l),$$
whenever $a\in \domain$ and  $1\leq k,l\leq n$.

Let $M_i=\range\xi_i\oplus\MC_i $ be the $\range$-$\range$ bimodule defined before. Then  
$$\M_i\otimes \mnc=\range\xi_i\otimes \mnc\oplus\MC_i\otimes \mnc$$ is a 
$\range\otimes \mnc$-$\range\otimes \mnc$ bimodule with the specified vector $ \xi_i\otimes I_n$  such that 
$$ (b_1\otimes e(k_1,l_1) )( m\otimes e(k,l))(b_2\otimes e(k_2,l_2)) =(b_1mb_2)\otimes(e(k_1,l_1)e(k,l)e(k_2,l_2)),$$
for $b_1,b_2\in \range$,  $m\in \M_i$ and $1\leq k_1,k_2,k,l_1,l_2,l\leq n.$

 Notice that  given two $\range$-$\range$ bimodules $\M_1$ and $\M_2$, we have that 
$$\Big(\M_1\otimes\mnc\Big)\otimes_{\range\otimes \mnc} \Big(\M_2\otimes\mnc\Big)= \Big(\M_1\otimes_\range \M_2\Big)\otimes_\C\mnc,$$
where $\Big(m_1\otimes e(k_1,l_1)\Big)\otimes_{\range\otimes \mnc} \Big(m_2\otimes e(k_2,l_2)\Big)$ is identified with $ \Big(m_1\otimes_\range m_2\Big)\otimes (e(k_1,l_1)e(k_2,l_2)).$ Therefore, the reduced free product of  $\Big(\M_i\otimes \mnc,\range\xi_i\otimes\mnc\Big)_{i\in \I}$ is 
$$\M\otimes \mnc= \range \xi \otimes \mnc\oplus \bigoplus\limits_{n=1}^\infty\bigoplus\limits_{i_1\neq i_2\neq \cdots \neq i_n} \MC_{i_1}\otimesb\cdots \otimesb\MC_{i_n}\otimes \mnc.$$
Let $(i_1,\cdots,i_n)\in \K(\I)$. Then  $P^{\M\otimes \mnc}_{(i_1,\cdots,i_n)}$  is the orthogonal projection from $\M\otimes \mnc$ onto the subspace 
$$ \MC_{i_1}\otimesb\cdots \otimesb\MC_{i_n}\otimes \mnc.$$

\begin{proposition}\label{matricial extension}\normalfont
Let $\I$ be an index set. Given an $\I$-compatible family $\J=\{\J_i\subseteq \K(\I)|i\in\I\}$.  Suppose that $\{\domain_i|i\in \I\}$ is  a $\range$-valued $\J$-independent family of subalgebras of a $\range$-valued probability space $(\domain,\E)$. Then, for each $n\in \N$, $\{\domain_i\otimes \mnc|i\in \I\}$ a $\range\otimes\mnc$-valued $\J$-independent family of subalgebras of the $\range\otimes \mnc$-valued probability space $(\domain\otimes \mnc,\E_n)$.
\end{proposition}
\begin{proof}
Let $\ii=(i_1,\cdots,i_m)\in \K$. Since  $P^{\M\otimes \mnc}_{\ii}$ is the orthogonal projection onto the subspace 
$$ \MC_{i_1}\otimesb\cdots \otimesb\MC_{i_m}\otimes \mnc$$
and   $P^{\M}_{\ii}$ is the orthogonal projection onto the subspace 
$$ \MC_{i_1}\otimesb\cdots \otimesb\MC_{i_m},$$
we have that  $P^{\M\otimes \mnc}_{\ii}=P^\M_{\ii}\otimes 1_{\mnc}$.\\
For $1\leq s\leq m$, let $\widetilde{a_{s}}=a_s\otimes e(k_s,l_s)\in \domain_{i_s}\otimes\mnc$, where $a_s\in \domain_{i_s}$.  We have 
\begin{align*}
\E_n[\widetilde{a_{1}}\cdots \widetilde{a_{m}}]&=\E[a_1\cdots a_n]\otimes (e(k_1,l_1)\cdots e(k_m,l_m))\\
&=\phi_\xi(P_{\J_{i_1}}\lambda_{i_1}(\gamma_{i_1}(a_1))\cdots P_{\J_{i_n}}\lambda_{i_m}(\gamma_{i_m}(a_m))) \otimes (e(k_1,l_1)\cdots e(k_m,l_m))\\
&=\langle \xi,P_{\J_{i_1}}\lambda_{i_1}(\gamma_{i_1}(a_1))\cdots P_{\J_{i_n}}\lambda_{i_m}(\gamma_{i_m}(a_m))\xi\rangle \otimes (e(k_1,l_1)\cdots e(k_m,l_m))\\
&=\langle \xi\otimes I_n,P_{\J_{i_1}}\lambda_{i_1}(\gamma_{i_1}(a_1))\otimes(e(k_1,l_1)\cdots P_{\J_{i_n}}\lambda_{i_m}(\gamma_{i_m}(a_m))\otimes e(k_m,l_m))(\xi\otimes I_n)\rangle\\
&=\langle \xi\otimes I_n,P^{(n)}_{\J_{i_1}}\lambda_{i_1,n}(\gamma_{i_1,n}(a_1)\otimes e(k_1,l_1)\cdots P^{(n)}_{\J_{i_m}}\lambda_{i_m}(\gamma_{i_m}(a_m)\otimes e(k_m,l_m))(\xi\otimes I_n)\rangle\\
&=\phi_{\xi\otimes I_n}( P^{(n)}_{\J_{i_1}}\lambda_{i_1,n}(\gamma_{i_1,n}(a_1)\cdots P^{(n)}_{\J_{i_m}}\lambda_{i_m,n}(\gamma_{i_m,n}(a_m)).
\end{align*} 

By linearity of the previous map, for each $s$,  $a_s\otimes e(k_s,l_s)$ can be replaced by an arbitrary element from $\domain_{i_s}$. The proof is done.
\end{proof}

\subsection{$k$-tuples of operator valued random variables and convolutions} In this subsection, we assume that $k$ is a positive integer.
Let $\range\langle X_1,\cdots, X_k\rangle$ be the set of noncommutative $\range$-valued polynomials in $k$ indeterminants, namely the linear span of $b_0X_{i_1}b_1X_{i_2}\cdots X_{i_n}b_n$ for $n\in\N\cup\{0\}$, $b_0,\cdots,b_n\in \range$ and $i_1,\cdots,i_n\in \{1,\cdots,k\}$.

\begin{definition}\normalfont
Let $x_1,\cdots,x_k$ be random variables from a $\range$-valued probability space $(\domain,\E)$. The joint distribution $\mu$ of $x_1,\cdots,x_k$ is a $\range$-linear map from $\range\langle X_1,\cdots, X_k\rangle$ to $\range$ defined as follows 
$$\mu[b_0X_{i_1}b_1X_{i_2}\cdots X_{i_n}b_n]=\E[b_0x_{i_1}b_1x_{i_2}\cdots x_{i_n}b_n].$$
\end{definition}

Let $X=\sum\limits_{l=1}^k x_l\otimes e(l,l)$. Then $X$ is a random variable in $(\domain \otimes \mnc,\E_n)$. 
One can easily see that the $\range\otimes\mnc$-valued distribution of $X$  is completely determined by the joint distribution of $(x_1,\cdots,x_k)$. On the other hand, we have that
\begin{align*}
&\E[x_{i_1}b_1\cdots b_{m-1}x_{i_m}] \otimes e(1,1)\\
=&(\E_n)[\Big((1_\domain\otimes e(1,i_1))X(1_\domain\otimes e(i_1,1))\Big) b_1\otimes e(1,1)  \cdots b_{m-1}\otimes e(1,1)\Big((1_\domain\otimes e(1,i_m))X(1_\domain\otimes e(i_m,1))\Big)].
\end{align*}

Therefore, the joint distribution of $x_1,\cdots,x_k$ can be decoded from the $\range\otimes\mnc$-valued distribution of $X$. Further, recall that the free additive convolutions  for $k$-tuples of random variables arise in the following way: 
Given joint distributions  $\mu, \nu$ of $k$-tuples of random variables, we can always find random variables $x_1,\cdots,x_k,y_1,\cdots,y_k$  in a $\range$-valued probability space $(\domain,\E)$ such that $\mu$ is equal to the joint distribution of $x_1,\cdots,x_k$ and $\nu$ is equal to the joint distribution of $y_1,\cdots,y_k$, and  the $\range$-algebra $\domain_1$ generated by $x_1,\cdots,x_k$ is freely independent from the $\range$-algebra $\domain_2$ generated by $y_1,\cdots,y_k$. 
Then, the free additive convolution $\mu\bp\nu$  is defined to be the joint distribution of $x_1+y_1,\cdots,x_k+y_k$. 
On the other hand, It is well know (or from  Proposition \ref{matricial extension}) that $\domain_1 \otimes\mnc$ is freely independent from 
$\domain_2\otimes\mnc$ in the $\range\otimes \mnc$-valued probability $(\domain\otimes\mnc,\E_n)$. 
It follows that $X=\sum\limits_{l=1}^k x_l\otimes e(l,l)$ and $Y=\sum\limits_{l=1}^k y_l\otimes e(l,l)$ are $\range\otimes \mnc$-freely independent in  $(\domain\otimes\mnc,\E\otimes I_n)$. 
Therefore,  the free additive convolution $\mu\bp\nu$  of $\mu$ and $\nu$ can be decoded from the free additive convolution of the $\range\otimes \mnc$-valued distributions $\mu_X$ and $\nu_Y$, where $\mu_X$ is $\range\otimes \mnc$-valued distribution of $X$ and $\nu_Y$ is $\range\otimes \mnc$-valued distribution of $X$. 

Now, we turn to introduce the additive convolutions associated with $\J$-independence.  
Even though the independence relations is well defined for subalgebras of a $\range$-valued probability, the $\J$-additive convolutions are not defined by $\J$-independence relations directly.

Let $\{\mu_i\}_{i\in I}$ be a family of $\range$-valued distributions, where $\I$ is a finite index set.  
Then, there are $\range$-valued operator valued probability spaces $(\domain_i,\E_i)_{i\in \I}$ and $x_i\in \domain_i$ such that for each $i$, $\mu_i$ is equal to the distribution of $x_i$. 
 For each $i$, let $(\M_i,\xi_i)$ be a $\range$-$\range$ bimodule with a specified vector such that there exits a $\range$-linear homomorphism $\gamma_i:\domain_i\rightarrow L(\M_i)$ and $\E_i[a]=\langle\xi_i, \gamma_i(a)\xi\rangle_{\M_i} $ for all $a\in\domain_i$.
 Let $(\M,\xi)$ be the reduced free product of $(\M_i,\xi)$ and $\lambda_i:L(\M_i)\rightarrow L(M)$ are the embedding isomorphisms.  
Given an $\I$-compatible family $\J=\{\J_i\subset \K(\I)|i\in \I\}$, then 
$$X=\sum\limits_{i\in\I} P^{\M}_{\J_i}\lambda_i(\gamma_i(x_i))\in L(\M).$$
Then we can defined $\mu$ to be the distribution of $X$ in $(L(\M),\langle \xi, \cdot \xi\rangle)$.  Now we show that $\mu$ is independent of the choices of $(\M_i,\xi_i)_{i\in \I}$. The proof follows Voiculescu's work in \cite{Voi1}.

\begin{definition}\normalfont
$\range$-morphisms  between $\range$-$\range$ bimodules with specified vector $(\M_1,\xi_1)$ and $(\M_2,\xi_2)$ are $\range$-$\range$-linear maps $S:\M\rightarrow\M'$, that is $S(b_1mb_2)=b_1S(m)b_2,$ 
for all $m\in \M_1$ and $b_1,b_2\in\range$,  such that
 $$S(\xi)=\xi'$$ and $$S(\MC_1)\subseteq {\MC_2}.$$
\end{definition}

\begin{lemma}\normalfont
Let $(\domain,\E)$ be a $\range$-valued probability space.   Let $(\M_1,\xi_1)$ and $ (\M_2,\xi_2)$ be $\range$-$\range$ bimodules with specified vectors such that there exist homomorphisms $\gamma_1:\domain\rightarrow L(\M_1)$ and $\gamma_2:\domain\rightarrow L(\M_2)$ such that 
$$\E[a]=\phi_{\xi_1}(\gamma_1(a))=\phi_{\xi_2}(\gamma_2(a)),$$
for all $a\in \domain.$ Then, there exists a $\range$-morphism $S:\M_1\rightarrow \M_2$ such that 
$$S\lambda_1(a)=\lambda_2(a)S$$
for $a\in \domain.$
\end{lemma}
\begin{proof}
Let $ \M_1'$ and $\M_2'$ be the closure of the sub-$\range$-$\range$ bimodules generated by $\lambda_1(\domain)\xi_1$ and $\lambda_2(\domain)\xi_2$, respectively. 
Since $\E[a]=\phi_{\xi_1}(\gamma_1(a))=\phi_{\xi_2}(\gamma_2(a))$
for all $a\in \domain,$  the map $S_0:\lambda_1(\domain)\xi_1\rightarrow\lambda_2(\domain)\xi_2$ defined by $S_0(\lambda_1(a)\xi_1)=\lambda_2(a)\xi_2$ is an isomorphism such that $S_0\xi_1=\xi_2$.
 Let $S$ be the extension of $S_0$ to $\M_1$ such that the restriction  $S|_{{\M_1'}^{\perp}}=0$. Then $S$ is the $\range$-morphism we want.
\end{proof}

Let $(\M'_i,\xi'_i)_{i\in I}$ be a family of $\range$-$\range$ bimodules with specified vectors such that for each $i$,   there exits a $\range$-linear homomorphism $\gamma'_i:\domain_i\rightarrow L(\M'_i)$ and $\E_i[a]=\phi_{\xi_i}(a)$ for all $a\in\domain_i.$
For each $i\in \I$, let $S_i$ be  a $\range$-morphism from $(\M_i,\xi_i)$ to $(\M'_i,\xi'_i)$ such that 
$$S\lambda_1(a)=\lambda_2(a)S,$$ 
for all $a\in\domain$.  
Let $(\M,\xi)$ and $(\M',\xi')$ be the reduced free product of $(\M_i,\xi_i)_{i\in I}$ and $(\M'_i,\xi'_i)_{i\in I}$, respectively. 
Then, we have a $\range$-morphism $S:(\M,\xi)\rightarrow (\M',\xi')$ defined as 
$$ S(b\xi)=b\xi', \quad \forall b\in\range$$
and $$S(m_1\otimesb\cdots\otimesb m_n)=S_{i_1}(m_1)\otimesb\cdots\otimes  S_{i_n}(m_n), $$
for all $n$, where $m_k\in\MC_{i_k}$ for $k=1,\cdots,n$. 
Let $\lambda'_i$ be the embedding isomorphism from $L(M_i)$. 
It follows the proof of Lemma 1.13 in \cite{Voi1}  that
$$ S\lambda_i(\gamma_i(a))= \lambda'_i(\gamma'_i(a))S.$$
One can easily check that $SP^{\M}_{\ii}=P^{\M'}_\ii S$ for all $\ii\in\K(\I)$. It follows that $SP^{\M}_{\J_i}=P^{\M'}_{\J_i}S$ for all $i\in \I$.  Therefore, we have that 
\begin{align*}
&SP^\M_{\J_{i_1}}\lambda_{i_1}(\gamma_{i_1}(a_1))\cdots P^\M_{\J_{i_m}}\lambda_{i_m}(\gamma_{i_m}(a_m))\xi\\
=&P^{\M'}_{\J_{i_1}}\lambda'_{i_1}(\gamma'_{i_1}(a_1))\cdots P^{\M'}_{\J_{i_m}}\lambda'_{i_m}(\gamma'_{i_m}(a_m))\xi'.
\end{align*}

According to the definition of $S$, we have that 
$$
\phi_{\xi}[P^\M_{\J_{i_1}}\lambda_{i_1}(\gamma_{i_1}(a_1))\cdots P^\M_{\J_{i_m}}\lambda_{i_m}(\gamma_{i_m}(a_m))]=\phi_{\xi'}[P^{\M'}_{\J_{i_1}}\lambda'_{i_1}(\gamma'_{i_1}(a_1))\cdots P^{\M'}_{\J_{i_m}}\lambda'_{i_m}(\gamma'_{i_m}(a_m))],
$$
whenever $a_k\in \domain_{i_k}$ for $1\leq k\leq m$.
Therefore,  $\{P^\M_{\J_{i}}\lambda_{i}(\gamma_{i}(x_i))|i\in \I \}$ and  $\{P^{\M'}_{\J'_{i}}\lambda'_{i}(\gamma'_{i}(x_i))|i\in \I \}$ have the same joint distribution.
It follows that the distribution of $X'=\sum\limits_{i\in\I} P^{\M'}_{\J_i}\lambda'_i(\gamma'_i(x_i))$ is also $\mu$. We have the following well-defined definition. 
\begin{definition}\label{definition of J-additive convolution}\normalfont
Given a finite index set $\I$, let $\{\mu_i\}_{i\in I}$ be a family of $\range$-valued distribution of $k$-tuples of random variables and $\J=\{\J_i\subset \K(\I)|i\in \I\}$ is $\I$-compatible.   
For each $i$, let $(\M_i,\xi_i)$ be a $\range$-$\range$ bimodule with a specified vector and $x(i,1),\cdots,x(i,k)\in\M_i$ such that $\mu_i$ is equal to the distribution of $x(i,1),\cdots,x(i,k)$.  
Let $(\M,\xi)$ be the reduced free product of $(\M_i,\xi)$ and $\lambda_i:L(\M_i)\rightarrow L(M)$ are the embedding isomorphisms.  
Then, the $\J$-additive convolution of $\{\mu_i\}_{i\in I}$ is the distribution of $(X_1,\cdots,X_k)$, where
$$X_l=\sum\limits_{i\in\I} P^{\M}_{\J_i}\lambda_i(x(i,l)),$$
for $1\leq l\leq k$.
\end{definition}

\begin{example}\normalfont Let $\I=\{1,2\}$ and $\J=\{\J_1,\J_2\}$ such that $\J_1=\J_2=\K(\I)$. 
Then the $\J$-additive convolution is the free additive  convolution, namely the $\J$-additive convolution of $\mu_1$ and $\mu_2$ is $\mu_1\bp\mu_2$. 
\end{example}

\begin{example}\normalfont Let $\I=\{1,2\}$ and $\J=\{\J_1,\J_2\}$ such that $\J_1=\{\emptyset,(1)\}$ and $\J_2=\{\emptyset,(2)\}$. Then the $\J$-additive convolution is the  Boolean  additive convolution, namely the $\J$-additive convolution of $\mu_1$ and $\mu_2$ is $\mu_1\uplus\mu_2$.
\end{example}

\begin{example}\normalfont Let $\I=\{1,2\}$ and $\J=\{\J_1,\J_2\}$ such that $\J_1=\{\emptyset,(1)\}$ and $\J_2=\{\emptyset,(1), (2),(2,1)\}$. Then the $\J$-additive convolution is the  monotone additive  convolution, namely the $\J$-additive convolution of $\mu_1$ and $\mu_2$ is $\mu_1\rhd\mu_2$.
\end{example}

In this paper, we will study  the following two $\J$-additive convolutions later.
\begin{definition}\normalfont\label{definition of orthogonal convolution}
  Let $\I=\{1,2\}$ and $\J=\{\J_1,\J_2\}$ such that $\J_1=\{\emptyset,(1)\}$ and $\J_2=\{(1), (2,1)\}$. Then the $\J$-additive convolution is called  orthogonal additive   convolution. The orthogonal additive convolution of $\mu_1$ and $\mu_2$ is denoted by $\mu_1\vdash\mu_2$.
\end{definition}

\begin{definition}\normalfont\label{definition of s-free convolution}
 Let $\I=\{1,2\}$ and $\J=\{\J_1,\J_2\}$ such that $\J_1=\{(i_1,\cdots,i_m)\in\K(\I))|i_m=1\}\cup\{\emptyset\}$ and $\J_2=\{(i_1,\cdots,i_m)\in\K(\I))|i_m=1\}$. Then $\J=\{\J_1,\J_2\}$ is $\I$-compatible and  the $\J$-additive convolution is called  s-free additive  convolution. The s-free additive convolution of $\mu_1$ and $\mu_2$ is denoted by $\mu_1\br\mu_2$.
\end{definition}

\begin{remark}\normalfont
The orthogonal additive convolution and the s-free additive convolution are neither commutative nor associative. 
\end{remark}

\begin{definition}\normalfont
Let $(x_1,\cdots,x_k)$ be a $k$-tuple of random variables in a $\range$-valued probability space $(\domain,\E)$. Then the matricial distribution of $(x_1,\cdots,x_k)$ is the $\range\otimes \mnc$-valued distribution $\mu_X$ of $X=\sum\limits_{l=1}^k x_l\otimes e(l,l)\in \domain \otimes \mnc$ in the $\range\otimes \mnc$-valued probability $(\domain\otimes\mnc,\E\otimes I_n)$.
\end{definition}

According the proof of  Proposition \ref{matricial extension}, we have the following property.
\begin{proposition}\label{Matricial convolution}\normalfont
Let   $\{\mu_i\}_{i\in I}$ be a family of $\range$-valued distributions of $k$-tuples of random variables  $(x(i,1),\cdots,x(i,k))\in(L(\M_i),\phi_{\xi_i})$  for each $i$ and $\mu$ is a $\range$-valued distribution of a $k$-tuple of random variables  $(x_1,\cdots,x_k)\in(L(\M),\phi_{\xi})$. 
Then the $\J$-additive convolution of $\{\mu_i\}_{i\in I}$ is equal to $\mu$ if and only if  the $\J$-additive convolution of the matricial distributions of $(x(i,1),\cdots,x(i,k))\in(L(\M_i),\phi_{\xi_i})$ is equal to the matricial distribution of $(x_1,\cdots,x_k)$.
\end{proposition}

\section{On the Boolean additive convolution and operator-valued Boolean convolution powers}
In this section, we study some analytic properties of the $\range$-valued Boolean additive convolution and the $\range$-valued Boolean convolution powers. 
The combinatorial aspects of  $\range$-valued Boolean convolution are studied in \cite{Po1}. 
\subsection{Boolean additive convolution}
\begin{lemma}\label{boolean product 1}\normalfont  Let $x,y$ be Boolean independent selfadjoint random variables from $(\domain, \E)$ and let$alg\{\range,x,y\}$ be  is the algebra generated by $\range,x,y$. Then 
$$\E[axf(y)]=\E[ax]\E[f(y)],$$
for all $a\in alg\{\range,x,y\}$ and $f\in \range\langle X \rangle.$
\end{lemma}
\begin{proof}
Notice that $ax$ is a linear combination of mononials ends with $x$, thus 
$$\E[axf_0(y)]=\E[ax]\E[f_0(y)],$$
for all $f_0\in \range_0\langle X \rangle. $

Suppose that $b\in \range$ is the constant term of $f$, namely $f-b\in \range_0\langle X \rangle $. Then we have 
$$ \E[ax(f(y)-b)+axb]=\E[ax]\E[f(y)-b]+\E[ax]b=\E[ax]\E[f(y)].$$
The proof is done.
\end{proof}

\begin{lemma}\label{boolean product 2}\normalfont
 Let $x,y$ be Boolean independent selfadjoint random variables from $(\domain, \E)$. Let  $b\in \range $ be such that $x+y-b$ and $y-b$ are invertible. Then
 $$\E[(b-x-y)^{-1}]=(1+\E[(b-x-y)^{-1}x])\E[(y-b)^{-1}].$$
\end{lemma}
\begin{proof}
Notice that $(b-x-y)^{-1} $ is contained in the norm closure of $alg\{\range,x,y\}$ and $(b-y)^{-1}$ in the norm closure of $alg\{\range,y\}$. By the the norm continuity of $\E$ and Lemma \ref{boolean product 1}, we have 
$$\E[(b-x-y)^{-1}x(b-y)^{-1}]=\E[(b-x-y)^{-1}x]\E[(b-y)^{-1}]. $$

On the other hand, we have   $$(b-x-y)^{-1}=(b-y)^{-1}+(b-x-y)^{-1}x(b-y)^{-1}.$$

It follows that 
\begin{align*}
\E[(b-x-y)^{-1}]&=\E[(b-y)^{-1}+(b-x-y)^{-1}x(b-y)^{-1}]\\
&=(1+\E[(b-x-y)^{-1}x])\E[(b-y)^{-1}].\\
\end{align*}
The proof is done.
\end{proof}

\begin{lemma}\normalfont
Let $\mu,\nu\in\cbd$. Then 
$$F_{\mu\uplus\nu,1}(b)= F_{\mu,1}(b)+F_{\nu,1}(b)-b,$$
for $b\in\mathbb{H}^+.$
\end{lemma}
\begin{proof}
Since $\mu,\nu\in\cbd$, there exist Boolean independent selfadjoint random variables $x,y$ from a $\range$-valued probability space $(\domain, \E)$ such that $\mu$ equals the distribution of $x$ and $\nu$ equals the distribution of $y$. 
By Lemma \ref{boolean product 2}, we have 
\begin{align*}
G_{\mu\uplus \nu,1}(b)F_{\nu,1}(b)&=\E[(b-x-y)^{-1}]\E[(b-y)^{-1}]^{-1}\\
&=(1+\E[(b-x-y)^{-1}x])\E[(b-y)^{-1}]\E[(b-y)^{-1}]^{-1}\\
&=1+\E[(b-x-y)^{-1}x].
\end{align*}
Similarly, we have $G_{\mu\uplus \nu,1}(b)F_{\mu,1}(b)= 1-\E[(b-x-y)^{-1}y]$.  Therefore,  we have 
\begin{align*}
&G_{\mu\uplus\nu,1}(b)(F_\mu(b)+F_{\nu,1}(b)-b)\\=&2+\E[(b-x-y)^{-1}x]+\E[(b-x-y)^{-1}y]-\E[(b-x-y)^{-1}b]\\
=&1.
\end{align*}
Since  $F_{\mu\uplus\nu,1}(b)=G_{\mu\uplus\nu,1}(b)^{-1}$, the statement follows.
\end{proof}

Notice that the Boolean additive convolution is a kind of $\J$additive convolution, the $\range\otimes\mnc$-valued distribution of $x\otimes I_n+y\otimes I_n$ is then the Boolean additive convolution of the distributions of $x\otimes I_n$ and $y\otimes I_n$.  Therefore, we have the following result.

\begin{proposition}\label{reciprocal Cauchy Boolean additive convolution} \normalfont
Let $\mu,\nu\in\cbd$. Then,  for all $n\in\N$, we have
$$F_{\mu\uplus\nu,n}(b_n)= F_{\mu,n}(b_n)+F_{\nu,n}(b_n)-b_n,$$
for $b_n\in\mathbb{H}^+_n.$ 
\end{proposition}

\subsection{Operator valued Boolean convolution Powers}
In this subsection, we study the reciprocal Cauchy transforms for  $\range$-valued Boolean convolution Powers. 
\begin{definition}\label{partition}\normalfont
Let $S$ be an ordered set. A partition $\sigma$ of a set $S$ is a collection of disjoint nonempty sets $V_1,...,V_r$ whose union is $S$.  $V_1,...,V_r$ are called  blocks of $\sigma$. A block $V$ of $\sigma$ is an interval if there is no triple $(s_1,s_2,r)$ such that $s_1<r<s_2$, $s_1,s_2\in V$, $r\not\in V$.  A partition $\sigma\in P(S)$ is an  interval partition if   every block of $\sigma$ is an interval.  The family of interval  partitions  of  $S$ will be denoted by $IN(S)$.   Let $l<n$ be two natural numbers, we denote by $[l,n]$ the interval $\{l,l+1,\cdots,k\}$ and denote by $[n]$ the interval $[1,n]$.
\end{definition}

\begin{definition}\normalfont
A $\range$-series $(\beta^{[n]})_{n\geq 1}$ is a sequence of $\C$-multilinear maps that, for each $n\geq 1$, $\beta^{[n]}:\range^{n-1}\rightarrow \range$. 
Given an interval partition $\sigma\in IN(n)$, then $\beta^{[\sigma]}$ is a $\C$-multilinear map form $\range^{n-1}$ to $\range$ defined recursively as follows:
 If $\sigma$ contains only one block $[n]$, then $\beta^{[\sigma]}=\beta^{[n]}$. Otherwise,  suppose that $[k+1,n]$ is an interval block of $\sigma$ and let $\sigma'$ be the restriction of  $\sigma$ to the interval $[1,k]$. Then
$$\beta^{[\sigma]}(b_1,\cdots,b_{n-1})=\beta^{[\sigma']}(b_1,\cdots,b_{k-1})b_k\beta^{[n-k]}(b_{k+1},\cdots,b_{n-1}).$$
\end{definition}

\begin{remark}\normalfont Notice that $\beta_1$ is always a  $\range$-constant.  
The partitions in the previous definition can be replaced by noncrossing partitions which are very useful combinatorial tools in studying free independence relation.
\end{remark}

\begin{definition}\normalfont
Given a random variable $x\in (\domain,\E)$, the moment series of  $x$ is a $\range$-series $(M_x^{[n]})_{n\geq 1}$ that $M_x^{[1]}=\E[x]$ and 
$$M_x^{[n]}(b_1,\cdots,b_{n-1})=\E[xb_1x\cdots xb_{n-1}x],$$ 
for $n\geq 2$ and $b_1,\cdots,b_n-1\in\range.$  
The B-transform of $(M_x^{[n]})_{n\geq 1}$ is the series $(B^{[n]}_x)_{n\geq 1}$, such that 
$$ M_x^{[n]}(b_1,\cdots,b_{n-1})=\sum\limits_{\sigma\in IN(n)} B^{[\sigma]}_x(b_1,\cdots,b_{n-1}).$$
\end{definition}

\begin{remark}\normalfont $(B^{[n]}_x)_{n\geq 1}$ can be solved recursively as follows: $B^{[1]}_x=\E[x]$ and 
$$ B_x^{[n]}(b_1,\cdots,b_{n-1})=\E[xb_1x\cdots xb_{n-1}x]-\sum\limits_{\sigma\in IN(n),\sigma\neq \{[n]\}} B^{[\sigma]}_x(b_1,\cdots,b_{n-1}),$$
for $n\geq 2$.
\end{remark}

\begin{definition}\normalfont Given $\mu\in \cbd$ such that  $\mu$ is equal to the distribution of $x$ in $(\domain,\E)$, then we set 
$$ B^{[\sigma]}_\mu=B^{[\sigma]}_x,$$ for all $\sigma\in IN(n)$, $n\in\N$. 
Let $\alpha$ be a completely positive map. Then  $\mu^{\uplus \alpha}$ is a  distribution such that 
$$B_{\mu^{\uplus \alpha}}^{[n]}=\alpha(B_{\mu}^{[n]}), $$
for all $n\in \N$.
\end{definition}

The following definition is from Popa \cite{Po1}, and is developed more in \cite{PV}.
\begin{definition}\normalfont Let $x\in (\domain,\E)$ be a selfadjoint random variable. The Boolean cumulants of $x$ are a family of $\range$-multilinear maps $\{B_{n,x}\}_{n\geq 1}$, such that $B_{n,x}:\range^n\rightarrow \range$  and 
$$\E[xb_1\cdots xb_n]=\sum\limits_{k=0}^{n-1}\E[xb_1\cdots xb_{k}]B_{n-k,x}(b_{k+1},\cdots,b_n).$$
\end{definition}

\begin{lemma}\label{equivalent boolean cumulants} \normalfont Given $x\in(\domain,\E)$, we have
$$B_{n,x}(b_1,\cdots,b_n)=B^{[n]}_x(b_1,\cdots,b_{n-1})b_n.$$
\end{lemma}
\begin{proof} When $n=1$,  we have $B_{1,x}(b_1)=\E[xb_1]=\E[x]b_1=B^{[1]}_xb_1$.\\
When $n\geq 2$, we have 
\begin{align*}
&B_{n,x}(b_1,\cdots,b_n)\\
=&\E[xb_1x\cdots b_{n-1}xb_n]-\sum\limits_{k=1}^{n-1}\E[xb_1\cdots xb_{k}]B_{n-k,x}(b_{k+1},\cdots,b_n)\\
=&\E[xb_1x\cdots b_{n-1}x]b_n-\sum\limits_{k=1}^{n-1}\E[xb_1\cdots b_{k-1}x]b_{k}B^{n-k}_{x}(b_{k+1},\cdots,b_{n-1})b_n\\
=&\left(\E[xb_1x\cdots b_{n-1}x]-\sum\limits_{k=1}^{n-1}\E[xb_1\cdots b_{k-1}x]b_{k}B^{[n-k]}_{x}(b_{k+1},\cdots,b_{n-1})\right)b_n\\
=&\left(\E[xb_1x\cdots b_{n-1}x]-\sum\limits_{k=1}^{n-1}\sum\limits_{\sigma'\in IN(k)}B^{[\sigma']}_x(b_1,\cdots,b_{k-1})b_kB^{[n-k]}_{x}(b_{k+1},\cdots,b_{n-1})\right)b_n\\
=&\left(\E[xb_1x\cdots b_{n-1}x]-\sum\limits_{\sigma \in IN(n), \sigma\neq \{[n]\}}B^{[\sigma]}_x(b_1,\cdots,b_{n-1}))\right)b_n\\
=& B^{[n]}_x(b_1,\cdots,b_{n-1})b_n.
\end{align*}
The proof is done.

\end{proof}

Let $(\M,\xi)$ be a  $\range$-$\range$ bimodule with a specified vector $\xi$, $x$ be a random variable in $(L(\M),\phi_\xi)$.
Let $\M=\range\xi\oplus \MC$ where $\MC=(\range\xi)^{\perp}.$ We write $x$ in the following matrix form
$$x=\left(\begin{array}{cc}
p&a\\
a^*&T
\end{array}\right),$$
where $a:\range\xi\rightarrow \MC$, $T:\MC\rightarrow \MC$, $p:\range\xi\rightarrow \range\xi$.

By Lemma 3.9 of \cite{PV} and Lemma \ref{equivalent boolean cumulants},  we have $B^{[1]}_x=\E[x]$ and
$$B^{[n]}_x(b_1,\cdots,b_{n-1})=B_{n,x}(b_1,\cdots b_{n-1},1_{\range})=\langle a^*b_1Tb_2\cdots Tb_{n-1}a\xi,\xi\rangle,$$
for $n\geq 2$.

 Let $\hh$ be a  $\range$-$\range$ Hilbert bimodule and $\zeta\in\hh$ be such that for $b\in\range$, we have
  $$ \langle \zeta,b\zeta \rangle_{\hh}=\alpha(b),$$ 
  where $ \langle\cdot ,\cdot\rangle_{\hh} $ is the inner product on $\hh$.
 
 Let $\K=\range\xi \oplus \MC\otimesb \hh$ with the inner product
$$\langle b\xi\oplus m\otimesb h ,b'\xi\oplus m'\otimesb h'\rangle_\K=b^*b'+\langle  h , \langle m, m'\rangle_\M h'\rangle_\hh,$$
for $b,b'\in \range$, $m,m'\in \MC$ and $h,h'\in \hh$.

Notice that $b\xi=\xi b$, the following maps are well defined. Let $\tilde{a},\tilde{p},\tilde{T}:\K \rightarrow \K$ be such that 
$$ \tilde{a}(b\xi\oplus m\otimesb h)=a(\xi)\otimesb (\zeta b), $$
$$ \tilde{p}(b\xi\oplus m\otimesb h)=\alpha(p)b\xi,$$
$$ \tilde{T}^*(b\xi\oplus m\otimesb h)=Tm\otimesb h,$$
for all $b\in \range$, $m\in \MC$ and $h\in \hh.$ 
\begin{lemma}\label{selfadjoint}\normalfont $\tilde{a}\in L(\K)$ and 
$$ \tilde{a}^*(b\xi\oplus m\otimesb h)=\langle \zeta,a^*(m)h\rangle_\hh\xi.$$
\end{lemma}
\begin{proof}
Notice that $\langle a(b\xi), m\rangle_\M= b^*a^*(m)$,  we have 
\begin{align*}
&\langle \tilde{a}(b\xi\oplus m\otimesb h), b'\xi\oplus m'\otimesb h'\rangle_\K\\
=&\langle a(\xi)\otimesb (\zeta b), b'\xi\oplus m'\otimesb h'\rangle_\K\\
=&\langle \zeta b, \langle a(b\xi),m'\rangle_M h'\rangle_\hh\\
=&b^*\langle \zeta , \langle a(\xi),m'\rangle_M h'\rangle_hh\\
=&b^*\langle \zeta ,  a^*(m') h'\rangle_hh\\
=&b^*\tilde{a}^*(b'\xi\oplus m'\otimesb h')\\
=&\langle b\xi\oplus m\otimesb h, \tilde{a}^*(b'\xi\oplus m'\otimesb h')\rangle_\K.
\end{align*}
The proof is done.
\end{proof}

 \begin{lemma}\normalfont  
 Let $\tilde{x}=\tilde{a}+\tilde{a}^*+\tilde{p}+\tilde{T}$. Then $\tilde{x}\in L(\K)$, namely $\tilde{x}$ is a selfadjointable operator. Moreover, $\tilde{x}=\tilde{x}^*.$
 \end{lemma}
\begin{proof}
It suffice to show that 
$$\langle \tilde{x}(b\xi\oplus m\otimesb h), b'\xi\oplus m'\otimesb h'\rangle_\K=\langle b\xi\oplus m\otimesb h, \tilde{x}(b'\xi\oplus m'\otimesb h')\rangle_\K,$$
for $b,b'\in \range$, $m,m'\in \MC$ and $h,h'\in H$.

Since $x$ is selfadjoint, we have $T$ is selfadjoint in $L(M)$ and 
\begin{align*}
\langle \tilde{T}(b\xi\oplus m\otimesb h), b'\xi\oplus m'\otimesb h'\rangle_\K&=\langle h,\langle m,T^*m' \rangle_\M h'\rangle_\hh
\\&=\langle h,\langle m,Tm' \rangle_\M h' \rangle_\hh\\
&=\langle b\xi\oplus m\otimesb h, \tilde{T} (b'\xi\oplus m'\otimesb h')\rangle_\K.
\end{align*}

Since $p\in \range$ is selfadjoint, we have  
\begin{align*}
\langle \tilde{p}(b\xi\oplus m\otimesb h), b'\xi\oplus m'\otimesb h'\rangle_\K&=b^*\alpha(p)^*b'\\
&=b^*\alpha(p)b'\\
&=\langle b\xi\oplus m\otimesb h, \tilde{p}(b'\xi\oplus m'\otimesb h')\rangle_\K.
\end{align*}

Together with Lemma \ref{selfadjoint}, the proof is done.
\end{proof}

\begin{proposition} \normalfont
Let $\tilde{\mu}$ be the distribution of $\tilde{x}$ with respect to $\langle \xi,  \xi \rangle_\K$. Then,
$$\tilde{\mu}=\mu^{\uplus \alpha},$$
where $\mu$ is the distribution of $x$ with respect to $\langle \xi,  \xi \rangle_\M$
\end{proposition}
\begin{proof}

By Lemma 3.9, we that 
$B^{[1]}_x=B_{1,x}(1_\range)=p$  and  $B^{[1]}_{\tilde{x}}=B_{1,\tilde{x}}(1_\range)=\alpha(p)=\alpha(B^{[1]}_x)$.

For $n\geq 2$,  we have 
\begin{align*}
B^{[n]}_x(b_1,\cdots,b_{n-1})&=B_{n,x}(b_1,\cdots,b_{n-1},1_\range)\\
&=\langle \xi,a^*b_1Tb_2T\cdots b_{n-1}T a\xi \rangle\\
&=a^*b_1Tb_2T\cdots b_{n-1}T a
\end{align*}
and 
\begin{align*}
B^{[n]}_{\tilde{x}}(b_1,\cdots,b_{n-1})&=B_{n,\tilde{x}}(b_1,\cdots,b_{n-1},1_\range)\\
&=\langle \xi, \tilde{a}^*b_1\tilde{T}b_2\tilde{T}\cdots b_{n-1}\tilde{T} \tilde{a}\xi\rangle_\K\\
&=\langle \xi, \tilde{a}^*b_1\tilde{T}b_2\tilde{T}\cdots b_{n-1}\tilde{T} \tilde{a}\xi\rangle_\K\\
&=\langle \xi \tilde{a}^*(b_1\tilde{T}b_2\tilde{T}\cdots b_{n-1}\tilde{T} a(\xi)\otimes\zeta) \rangle_\K\\
&=\langle \xi \tilde{a}^*[(b_1Tb_2T\cdots b_{n-1}T a(\xi))\otimes\zeta]\rangle_\K\\
&=\langle \xi, \langle \zeta, a^*(b_1Tb_2T\cdots b_{n-1}T a(\xi))\zeta\rangle_\hh\rangle_\K\\
&=\langle \xi, \alpha(a^*b_1Tb_2T\cdots b_{n-1}T a)\xi\rangle_\K\\
&= \alpha(a^*b_1Tb_2T\cdots b_{n-1}T a)\\
&= \alpha(B^{[n]}_x(b_1,\cdots,b_{n-1})).\\
\end{align*}
The proof is done.

\end{proof}

For $b\in\mathbb{H}^+$, by Proposition \ref{reciprocal Cauchy and representation}, we have that 
$F_{\mu_x,1}(b)=b-p-a^*(b-T)^{-1}a$, thus 
\begin{align*}
F_{\tilde{\mu},1}(b)&=\tilde{p}-b-\tilde{a} ^*(\tilde{T}-b)^{-1}\tilde{a}\\
&=b-\tilde{p}-\langle \xi, \tilde{a} ^*(b-\tilde{T})^{-1}\tilde{a}\xi\rangle_\K\\
&=b-\tilde{p}-\langle a(\xi)\otimes\zeta, (b-\tilde{T})^{-1}(a(\xi)\otimes\zeta)\rangle_\K\\
&=b-\tilde{p}-\langle a(\xi)\otimes\zeta,( (b-T)^{-1}a(\xi))\otimes\zeta\rangle_\K\\
&=b-\tilde{p}-\langle \zeta,\langle a(\xi), ( (b-T)^{-1}a(\xi))\rangle_\M\zeta\rangle_\hh\\
&=b-\alpha(p)-\alpha(\langle a(\xi), ( (b-T)^{-1}a(\xi))\rangle_\M)\\
&=b-\alpha(p)-\alpha(a^*(T-b)^{-1}a)\\
&=\alpha(F_{\mu,1}(b))+b-\alpha(b).
\end{align*}
Therefore, we have the following result.
\begin{lemma}\normalfont
 Let $\mu$  be a $\range$-valued distribution and let $\alpha$ be a completely positive map from $\range$ to $\range$.  Then, 
 $$F_{\mu^{\uplus \alpha},1}(b)=\alpha(F_{\mu,1}(b))+(1-\alpha)(b), $$
 for all $b\in\mathbb{H}^+$.
\end{lemma}

Similarly following  the previous construction,  we have the following result.

\begin{theorem}\label{reciprocal Cauchy transform of Boolean operator power convolution}\normalfont
 Let $\mu\in\cbd$  be a $\range$-valued distribution and let $\alpha$ be a completely positive map from $\range$ to $\range$.  For each $n\in\N$, we have 
 $$F_{\mu^{\uplus \alpha,n}}(b_n)=\alpha_n(F_{\mu,n}(b_n))+(1_n-\alpha_n)(b_n), $$
 for all $b\in\mathbb{H}_n^+$ and $\alpha_n=\alpha\otimes I_n$ is a map from $\range\otimes\mnc$ to $\range\otimes \mnc$ such that $\alpha_n(b\otimes e)=\alpha(b)\otimes e$ for all $b\in \range$ and $e\in \mnc$.
\end{theorem}

\section{Monotone Convolutions} 
In this section, we study some analytic properties of the $\range$-valued monotone additve convolution. The combinatorial aspects of  $\range$-valued monotone additive convolution are studied in \cite{Po2}. Given $x\in (\domain,\E)$, we denote by $alg_0\{x,\range\}=\{p(x)|p\in\range\langle X\rangle_0\}$ and $alg\{x,\range\}=\{p(x)|p\in\range\langle X\rangle\}$.

\begin{lemma}\normalfont
Let $x,y\in(\domain,\E)$ be two $\range$-valued random variables such that $x,y$ are monotone. Then $alg_0\{x,\range\}$ and $alg\{y,\range\}$ are monotone independent.
\end{lemma}
\begin{proof} Let $a_1,\cdots, a_n\in alg_0\{x,\range\}\cup alg\{y,\range\}$ such that $a_k\in alg\{y,\range\}$ and $a_{k-1},a_{k+1}\in alg_0\{x,\range\}$.
Then $a_k=a'_k+b$ for some $a_k'\in alg_0\{y,\range\}$ and $b\in \range$. Since $x,y$ are monotone independent, we have 
$$\E[a_1\cdots a_{k-1}a'_ka_{k+1} a_n]=\E[[a_1\cdots a_{k-1}\E[a'_k]a_{k+1} a_n].$$
On the other hand, 
$$\E[a_1\cdots a_{k-1}ba_{k+1} a_n]=\E[[a_1\cdots a_{k-1}\E[b]a_{k+1} a_n].$$
Therefore,
$$\E[a_1\cdots a_{k-1}a_ka_{k+1} a_n]=\E[[a_1\cdots a_{k-1}\E[a_k]a_{k+1} a_n].$$
Since $a_1,\cdots, a_n\in alg_0\{x,\range\}\cup alg\{y,\range\}$ are arbitrary, the proof is done.
\end{proof}

\begin{lemma}\normalfont
Let $x,y\in(\domain,\E)$ be two $\range$-valued random variables such that $x,y$ are monotone. Given $a\in alg(x,y,\range), x_1\in alg_0\{x, \range\}, x_2\in alg\{x, \range\}$ and $y'\in alg\{y,\range\}$, then 
$$\E[ax_1y'x_2]=\E[ax_1\E[y']x_2].$$
\end{lemma}
\begin{proof}
Since $x_2\in alg\{x, \range\}$, there exists $x_2'\in alg_0\{x, \range\}$ such that $x_2=x_2'+b$ for some $b\in \range$. Therefore,
\begin{align*}
  \E[ax_1y'x_2]
=&\E[ax_1 y'(x_2'+b)]\\
=&\E[ax_1 y'x_2']+\E[ax_1yb]\\
=&\E[ax_1 \E[y']x_2']+\E[ax_1]\E[y'b]\\
=&\E[ax_1 \E[y']x_2']+\E[ax_1\E[y']b]\\
=&\E[ax_1\E[y']x_2],
  \end{align*}
which is desired.
\end{proof}

\begin{lemma}\label{Monotone property 1}\normalfont Let $x,y\in(\domain,\E)$ be two $\range$-valued selfadjoint random variables such that $x,y$ are monotone.  For  $b\in \mathbb{H}^+$, we have  
$$\E[(b-y)(b-x-y)^{-1}])=\E[(1-xG_y(b))^{-1}].$$
\end{lemma}

\begin{proof}
Since $x,y$ are selfadjoint and $b\in \mathbb{H}^+$, $b-y$  and $b-x-y$ are invertible.  It follows that  
$$ \E[(b-y)(b-x-y)^{-1}])=\E[(1-x(b-y)^{-1})^{-1}]).$$
Therefore, we have 
 \begin{align*}
 & \E[(1-x(b-y)^{-1})^{-1}])-\E[(1-xG_y(b))^{-1}]\\
=& \E[(1-x(b-y)^{-1})^{-1}(xG_y(b)-x(b-y)^{-1})(1-xG_y(b))^{-1}]\\
=& \E[(1-x(b-y)^{-1})^{-1}(x(G_y(b)-(b-y)^{-1}))(1-xG_y(b))^{-1}].
  \end{align*}

Notice that $(1-xG_y(b))^{-1}\in \overline{alg\{x,\range\}}^{\|\cdot\|}$, $(G_y(b)-(b-y)^{-1})\in \overline{alg\{y,\range\}}^{\|\cdot\|}$, $\E[(G_y(b)-(b-y)^{-1})]=0$ and $\E$ is norm-continuous, we have that 
\begin{align*}
 & \E[(1-x(b-y)^{-1})^{-1}x(G_y(b)-(b-y)^{-1})(1-xG_y(b))^{-1}]\\
=&\E[(1-x(b-y)^{-1})^{-1}x\E[G_y(b)-(b-y)^{-1}](1-xG_y(b))^{-1}]\\
=&0.
  \end{align*}
  The proof is done.
\end{proof}

\begin{lemma}\normalfont
Let $\mu,\nu\in\cbd$. Then,
$$ F_{\mu\rhd \nu,1}(b)=F_{\mu,1}(F_{\nu,1}(b)),$$
 for  $b\in \mathbb{H}^+$.
\end{lemma}
\begin{proof}
Let $x,y$ be monotone independent selfadjoint random variables  from a $\range$-valued probability space $(\domain, \E)$ such that $\mu$ equals the distribution of $x$ and $\nu$ equals the distribution of $y$. 
By Lemma \ref{Monotone property 1}, we have 
 \begin{align*}
  G_{\mu\rhd \nu,1}(b)&=\E[(b-x-y)^{-1}]\\
&= \E[(b-y)^{-1}] +\E[(b-y)^{-1}x(b-x-y)^{-1}] \\
&=\E[(b-y)^{-1}] +\E[(b-y)^{-1}]\E[x(b-x-y)^{-1}] \\
&=\E[(b-y)^{-1}] [1_{\range}+\E[x(b-x-y)^{-1}] \\
&=\E[(b-y)^{-1}] (1_{\range}+\E[x(b-x-y)^{-1}])\\
&=\E[(b-y)^{-1}] (\E[(b-y)(b-x-y)^{-1}])\\
&=G_{\nu,1}(b)\E[(1-xG_{\nu,1}(b))^{-1}]\\
&=\E[(G_{\nu,1}(b)^{-1}-x)^{-1}]\\
&=G_{\mu,1}(F_{\nu,1}(b)).
      \end{align*}
The proof is done.
\end{proof}

Notice that monotone  additive convolution is a kind of $\J$- additive convolution, the  $\range\otimes\mnc$-valued distribution of $x\otimes I_n+y\otimes I_n$ is then the monotone additive convolution of the distributions of $x\otimes I_n$ and $y\otimes I_n$.  Therefore, we have the following result.

\begin{proposition}\label{reciprocal Cauchy Monotone additive convolution}\normalfont
Let $\mu,\nu\in\cbd$. For $n\in \N$, we have 
$$F_{\mu\rhd \nu,n}(b_n)=F_{\mu,n}(F_{\nu,n}(b_n)),$$
 for  $b_n\in \mathbb{H}_n^+$. 
\end{proposition}

\section{Orthogonal Convolution}In this section, we study some analytic properties of the orthogonal additive convolutions.
Let $\mu_1,\mu_2\in\cbd$. 
For $i=1,2$, let $(\M_i,\xi_i)$ be  a $\range$-$\range$ bimodules with  specified vectors and $x_i\in (L(\M_i),\phi_{\xi_i})$ such that $\mu_i$ is equal to the distribution of $x_i$.  Let $(\M,\xi)$ be the reduced free product of $(\M_i,\xi)$ and $\lambda_i:L(\M_i)\rightarrow L(M)$ are the  embedding isomorphisms defined in Section 3.  Let $\J_1=\{\emptyset,(1)\}$ and $\J_2=\{(1), (2,1)\}$. Then, $P^{\M}_{\J_1}$ is the orthogonal projection onto 
$$\range_\xi\oplus\MC_1$$
and 
$P^{\M}_{\J_1}$ is the orthogonal projection onto 
$$ \MC_1\oplus\MC_2\otimesb\MC_1.$$
Then, the orthogonal additive convolution $\mu_1\vdash\mu_2$ of $\mu$ and $\nu$ is the distribution of $P^{\M}_{\J_1}\lambda_1(x_1)+P^{\M}_{\J_2}\lambda_2(x_2)$ in $(L(\M),\phi_{\xi})$.

Notice that  $P^{\M}_{\J_1}$ and $P^{\M}_{\J_2}$ are projections from $\M$ into
$$\M_{\vdash}=\range\xi\oplus\MC_1\oplus\MC_2\otimesb\MC_1$$
and $(\M_{\vdash},\xi)$ is a $\range$-$\range$ bimodules with  specified vector,
it follows that the distribution of $P^{\M}_{\J_1}\lambda_1(x_1)+P^{\M}_{\J_2}\lambda_2(x_2)$ in $(L(\M),\phi_{\xi})$ is equal to 
the distribution of $P^{\M}_{\J_1}\lambda_1(x_1)+P^{\M}_{\J_2}\lambda_2(x_2)$ in $(L(\M_\vdash),\phi_{\xi})$.

On the other hand, let $\M_{\rhd}=\range\xi\oplus\MC_1\oplus \MC_2\oplus\MC_2\otimesb\MC_1$. Then $\M_{\rhd}$ is also a $\range$-$\range$ bimodules with  specified vector such that 
$$ \MC_{\rhd}=\MC_2\oplus \MC_{\vdash}.$$
Therefore,  $L(\M_2)$ and $L(\M_{\vdash})$ are boolean independent in $(L(\M_{\rhd}),\phi\xi)$. Let $\J_2'=\{\emptyset,(2)\}$. Then $P^{\M}_{\J_2'}$ is the orthogonal projection onto $\range\xi\oplus\MC_2$. It follows that 
the distribution of 
$$P^{\M}_{\J_2'}\lambda_2(x_2)+P^{\M}_{\J_1}\lambda_1(x_1)+P^{\M}_{\J_2}\lambda_2(x_2) $$
is the Boolean convolution of $\mu_2$ and $\mu_1\vdash\mu_2$.
Notice that $P^{\M}_{\J_2'}+P^{\M}_{\J_2}=P^{\M}_{\J''_2}$ where $\J''_2=\{\emptyset,(1),(2),(1,2)\}$, and $\J'=\{\J_1,\J''_2\}$-additive convolution of $\mu_1$ and $\mu_2$ is $\mu_1\rhd\mu_2$. Therefore, we have the following result.
\begin{proposition}\normalfont\label{orthogonal, boolean and monotone}Let $\mu,\nu\in \cbd$. Then
$$\mu\rhd \nu=\nu \uplus (\mu\vdash \nu ).$$
\end{proposition}

\begin{proposition}\normalfont\label{reciprocal cauchy of orthogonal} Given $\mu,\nu\in \cbd$,  for each $n\in\N$, we have
$$F_{\mu\vdash\nu,n}(b_n)=F_{\mu,n}(F_{\nu,n}(b_n))-F_{\nu,n}(b_n)+b_n,$$
for all $b_n\in \mathbb{H}_n^+.$
\end{proposition}
\begin{proof}
By Proposition \ref{orthogonal, boolean and monotone}, Proposition \ref{reciprocal Cauchy Boolean additive convolution} and Proposition \ref{reciprocal Cauchy Monotone additive convolution}, for each $n\in\N$, we have
$$F_{\mu,n}(F_{\nu,n}(b_n))=F_{\mu\vdash\nu,n}(b_n)+F_{\nu,n}(b_n)-b_n,$$
for all $b_n\in \mathbb{H}_n^+.$
The proof is done.
\end{proof}

\section{Subordination convolution}
\subsection{S-free additive convolution}

In this section, we study analytic properties of the s-free additive convolutions.
For $i=1,2$, let $\mu_i,\in\cbd$, 
$(\M_i,\xi_i)$ be  a $\range$-$\range$ bimodules with  specified vectors and $x_i\in (L(\M_i),\phi_{\xi_i})$ such that $\mu_i$ is equal to the distribution of $x_i$.  Let $(\M,\xi)$ be the reduced free product of $(\M_i,\xi_i)$ and $\lambda_i:L(\M_i)\rightarrow L(\M)$ are the  embedding isomorphisms defined in Section 3. 
Let $\J_1=\{(i_1,\cdots,i_m)\in\K(\I))|i_m=1\}\cup\{\emptyset\}$ and $\J_2=\{(i_1,\cdots,i_m)\in\K(\I))|i_m=1\}$. 
Then, $P^{\M}_{\J_1}$ is the orthogonal projection onto 
$$\range\xi\oplus\MC_1\oplus\MC_2\otimesb \MC_1\oplus\MC_1\otimesb \MC_2\otimesb\MC_1\cdots$$
and 
$P^{\M}_{\J_2}$ is the orthogonal projection onto 
$$\MC_1\oplus\MC_2\otimesb \MC_1\oplus\MC_1\otimesb \MC_2\otimesb\MC_1\cdots$$
Then, the s-free additive convolution $\mu_1\br\mu_2$ of $\mu, \nu$ is equal to the distribution of $X=P^{\M}_{\J_1}\lambda_1(x_1)+P^{\M}_{\J_2}\lambda_2(x_2)$ in $(L(\M),\phi_{\xi})$.

Let $\M_{\br}=\range\xi\oplus \MC_{br}$, where
$$  \MC_{br}=\MC_1\oplus\MC_2\otimes \MC_1\oplus\MC_1\otimes \MC_2\otimes\MC_1\cdots.$$
Then, we have that 
\begin{align*}
\M=\range\xi\oplus\MC_{br}\oplus\MC_2\oplus\MC_{br}\otimes\MC_2
\end{align*}
and there exits a unitary $V:\M_{\br}\otimes\M_2\rightarrow \M$
such that 
$$\begin{array}{cl}
V:& \xi\otimesb\xi_2\rightarrow \xi,\\
V:&\xi\otimesb m_2\rightarrow m_2,\\
V:& m_1\otimesb \xi_2\rightarrow m_1,\\
V:& m_1\otimesb m_2 \rightarrow  m_1\otimesb m_2,
\end{array} $$ 
where $m_1\in \MC_\br$, $m_2\in \MC_2$.

Let $\lambda_2(x_2)|_{\M_2}$ be the restriction of $\lambda_2(x_2)$ onto $\range\xi\oplus \MC_2$. Then, the distribution of  $V(X\otimesb 1_{\M_2})V^{-1}+\lambda_2(x_2)|_{\M_2}$ is $\mu_2\rhd\mu_X=\mu_2\rhd (\mu_1\boxright \mu_2).$
\begin{lemma}
$$V(X\otimesb 1_{\M_2})V^{-1}+\lambda_2(x_2)|_{\M_2}=\lambda_1(x_1)+\lambda_2(x_2).$$
\end{lemma}
\begin{proof}It suffices to show that 
$$ [V(X\otimesb 1_{\M_2})V^{-1}+\lambda_2(x_2)|_{\M_2}](v)=[\lambda_1(x_1)+\lambda_2(x_2)](v)$$
for all $v\in\M.$ 

We have the following cases:

Case 1:  $v\in\range\xi$, then $v=b\xi$ for some $b\in \range$.  We have 
\begin{align*}
V(X\otimes 1_{\M_2})V^{-1}b\xi&=V(X\otimes 1_{\M_2})(b\xi\otimes \xi_2).
\end{align*}
and 
$$
X(b\xi)=(P^{\M}_{\J_1}\lambda_1(x_1)+P^{\M}_{\J_2}\lambda_2(x_2))(b\xi)=\lambda_1(x_1)b\xi\\
$$
and 
$$\lambda_2(x_2)|_{\M_2}(b\xi)=\lambda_2(x_2)b\xi.$$
The equality holds in this case.

Case 2:  $v\in \MC_2$, then
\begin{align*}
V(X\otimes 1_{\M_2})V^{-1}v&=V(X\otimes 1_{\M_2})(\xi\otimes v)\\ 
&=V(\lambda_1(x_1)\xi\otimes v)\\ 
&=\phi_{\xi}(\lambda_1(x_1))v +[\lambda_1(x_1)\xi-\phi_{\xi}(\lambda_1(x_1))]\otimes v\\ 
&= \lambda_1(x_1)v.
\end{align*}
and 
$$\lambda_2(x_2)|_{\M_2}(v)=\lambda_2(x_2)v.$$
The equality holds in this case.

Case 3:  $v\in \MC_\br$, then
$$V(X\otimes 1_{\M_2})V^{-1}v=V(X\otimes 1_{\M_2})(v\otimes \xi_2)=Xv$$
and 
$$P^{\M}_{\J_1}v=P^{\M}_{\J_2}v=v.$$
Therefore,  $Xv=[ \lambda_1(x_1)+\lambda_2(x)]v$.
Ont the other hand, $\lambda_2(x_2)|_{\M_2}v=0$.
The equality holds in this case.

Case 3:  $v\in \MC_\br\otimes\MC_2$, then $v=v_1\otimes v_2$ where $v_1\in \MC_\br$ and $v_2\in \MC_2$.
Similarly, we have 
$$Xv_1=[ \lambda_1(x_1)+\lambda_2(x)]v_1$$
and
\begin{align*}
V(X\otimes 1_{\M_2})V^{-1}v&=V(X\otimes 1_{\M_2})(v_1\otimes v_2)\\
&=V([\lambda_1(x_1)+\lambda_2(x)](v_1)\otimes v_2)\\
&=[ \lambda_1(x_1)+\lambda_2(x)](v_1\otimes v_2)
\end{align*}
and $\lambda_2(x_2)|_{\M_2}]v=0$.
The equality holds in this case.

The proof is done.
\end{proof}

Therefore, the distribution of  $V(X\otimes 1_{\M_2})V^{-1}+\lambda_2(x_2)|_{\M_2} $ is $\mu_1\bp\mu_2$ and   we have the following result.
\begin{proposition}\label{subordination, monotone}\normalfont Let $\mu_1,\mu_2\in\cbd$ be two $\range$-valued distributions. Then, we have  
$$\mu_2\rhd (\mu_1\boxright \mu_2)=\mu_1\boxplus \mu_2. $$
\end{proposition}

From the previous construction,  we have the following generalization of  Proposition 2.8 in \cite{BMS}. 

\begin{proposition}\label{reciprocal cauchy of Free-Monotone}\normalfont   Let $\mu,\nu\in\cbd$.  For each $n\in \N$. we have 
$$F_{\mu\bp\nu,n}(b_n)=F_{\mu,n}(F_{\nu\br\mu,n}(b_n))$$
for all $b_n\in \mathbb{H}_n^+.$
\end{proposition}
\begin{proof}
The statement directly follows  Proposition \ref{subordination, monotone} and Proposition \ref{reciprocal Cauchy Monotone additive convolution}.
\end{proof}

Since the $\range$-valued reciprocal Cauchy transform of an $\range$-valued distribution  is one to one from a neighborhood of infinity in $\mathbb{H}^+$ to a neighborhood of infinity in $\mathbb{H}^+$ \cite{Wi1}, $F_{\nu\boxright\mu,n}$ is uniquely determined by
$$ F^{\langle-1 \rangle}_{\mu,n}(F_{\mu\boxright\nu,n}(b_n))=F_{\mu,n}(F_{\nu\boxright\mu,n}(b_n)),$$
 where $b_n$ is from a neighborhood of infinity in $\mathbb{H}_n^+.$

Now $\widetilde{\J_1}=\{(i_1,\cdots,i_m)\in\K(\I))|i_m=2\}$ and $\widetilde{\J_2}=\{(i_1,\cdots,i_m)\in\K(\I))|i_m=2\}\cup\{\emptyset\}$ Then, $P^{\M}_{\widetilde{\J_1}}$ is the orthogonal projection onto 
$$\MC_2\oplus\MC_1\otimesb \MC_2\oplus\MC_2\otimesb \MC_1\otimesb\MC_2\cdots$$
and 
$P^{\M}_{\widetilde{\J_2}}$ is the orthogonal projection onto 
$$\range_\xi\oplus\MC_2\oplus\MC_1\otimesb \MC_2\oplus\MC_2\otimesb \MC_1\otimesb\MC_2\cdots$$
Then, the s-free additive convolution $\mu_2\br\mu_1$ of $\mu, \nu$ is equal to the distribution of $\widetilde{X}=P^{\M}_{\widetilde{\J_1}}\lambda_1(x_1)+P^{\M}_{\widetilde{\J_2}}\lambda_2(x_2)$ in $(L(\M),\phi_{\xi})$.

Let $\widetilde{\M_{\br}}=\range\xi\oplus \widetilde{\MC_{br}}$, where
$$  \widetilde{\MC_{br}}=\MC_2\oplus\MC_1\otimesb \MC_2\oplus\MC_2\otimesb \MC_1\otimesb\MC_2\cdots.$$
Then we have 
$$\M=\xi\oplus\widetilde{\MC_{\br}}\oplus\MC_\br.$$
Therefore, the distribution of $X+\widetilde{X}$ is the Boolean additive convolution of $(\mu_1\br\mu_2)\uplus(\mu_2\br\mu_1)$. Notice that 
$$X+\widetilde{X}=\lambda_1(x_1)+\lambda_2(x_2).$$

We have the following decomposition for the free additive convolutions.

\begin{proposition}\label{subordination, Boolean}\normalfont Let $\mu_1,\mu_2\in\cbd$ be two $\range$-valued distributions. Then, we have  
$$\mu_1\bp\mu_2=(\mu_1\br \mu_2)\uplus(\mu_2\br \mu_1). $$
\end{proposition}

\begin{proposition}\label{reciprocal cauchy of free-Boolean}\normalfont Let $\mu,\nu\in\cbd$.  For each $n\in \N$, we have 
$$F_{\mu\bp\nu,n}(b_n)=F_{\nu\br\mu,n}(b_n)+F_{\mu\br\nu,n}(b_n)-b_n$$
for all $b_n\in \mathbb{H}_n^+.$
\end{proposition}
\begin{proof}
The statement directly follows  Proposition \ref{subordination, Boolean} and Proposition \ref{reciprocal Cauchy Boolean additive convolution}.
\end{proof}

\subsection{S-free additive convolution and operator-valued free convolution powers}
In this subsection we study relations between s-free additive convolution and the subordination functions of the  operator valued free convolution powers defined in \cite{ABFN}.
Let $\mu\in\cbd$ and $\alpha$ be a completely positive map such that $\eta=\alpha-1$ is also completely positive.  By Theorem 8.4 in \cite{ABFN} and the matricial extension property of the free convolution, for each $n\in\N$, we have that there is a unique solution $\omega_{\alpha,\mu,n}$ to  the functional equation
\begin{equation}\label{operator power subordination} \omega_{\alpha,\mu,n}(b_n)=b_1+(\alpha_n-1_n)h_{\mu,n}(\omega_{\alpha,\mu,n}(b_n)),\quad b_n\in\mathbb{H}_n^+,
\end{equation} 
where $h_{\mu,n}(b_n)=F_{\mu,n}(b_n)-b_n$.  

\begin{proposition}\normalfont\label{reciprocal Cauchy transform of  s-free convolution powers}
  Let $\mu\in\cbd$. For each $n\in \N$,  we have
  $$\omega_{\alpha,\mu,n}(b_n)=F_{{\mu^{\uplus\eta}\boxright \mu^{\uplus\eta}}}(b_n).$$
\end{proposition}
\begin{proof}
It suffices to show  that $F_{{\mu^{\uplus\eta}\boxright \mu^{\uplus\eta}}}$ satisfies Equation (\ref{operator power subordination}).  By Equation (\ref{operator power subordination}), we have 
\begin{equation}\label{operator power subordination 2} 
\alpha_n(\omega_{\alpha,\mu,n}(b_n))=b_n+(\alpha_n-1_n)F_\mu(\omega_{\alpha,\mu,n}(b_n)),\quad b_n\in\mathbb{H}_n^+.
\end{equation} 

By Proposition \ref{reciprocal cauchy of free-Boolean}, we have 
\begin{equation}
2F_{{\mu^{\uplus\eta}\boxright \mu^{\uplus\eta},n}}(b_n)-b_n=F_{\mu^{\uplus\eta},n}(F_{{\mu^{\uplus\eta}\boxright \mu^{\uplus\eta},n}}(b_n)),\quad b_n\in\mathbb{H}^+.
\end{equation} 

Then, by Theorem \ref{reciprocal Cauchy transform of Boolean operator power convolution}, we have 
\begin{equation}
2F_{{\mu^{\uplus\eta}\boxright \mu^{\uplus\eta},n}}(b_n)-b_n=\eta_n[F_{\mu,n}(F_{{\mu^{\uplus\eta}\boxright \mu^{\uplus\eta},n}}(b_n))]+(1_n-\eta_n)F_{{\mu^{\uplus\eta}\boxright \mu^{\uplus\eta},n}}(b_n),\quad b_n\in\mathbb{H}^+.
\end{equation} 
Therefore, we have 
$$ (1+\eta)F_{{\mu^{\uplus\eta}\boxright \mu^{\uplus\eta},n}}(b_n)-b_n=\eta_n[F_{\mu,n}(F_{{\mu^{\uplus\eta}\boxright \mu^{\uplus\eta},n}}(b_n))].$$
Notice that $1_n+\eta_n=\alpha_n$,  $F_{{\mu^{\uplus\eta}\boxright \mu^{\uplus\eta},n}}$ is a solution to Equation \ref{operator power subordination 2}. 
It follows that $\omega_{\alpha,\mu,n}=F_{{\mu^{\uplus\eta}\boxright \mu^{\uplus\eta},n}}$ on $\mathbb{H}^+_n.$
\end{proof}

\subsection{Matricial R-transform and the free additive s-free additive convolution}  In this subsection, we prove an operator valued version of Relations (\ref{scalar free addtive and subordination}) and (\ref{scalar free powers and subordination}). 
\begin{lemma}\normalfont Given  $\mu,\nu\in \cbd$, let $\mu\br\nu$ be the $s$-free additive convolution of $\mu,\nu$ and let  $(R_{\mu,n})_{n\geq 1}$ and $(R_{\mu\br\nu,n})_{n\geq 1}$ be the matricial R-transform of $\mu$ and $\mu\br\nu$, respectively. Then
 for each $n\in\N$, we have that 
$$  R_{\mu\br\nu,n}(b_n^{-1})=R_{\mu,n}(G_{\nu,n}(b_n)), $$
where $b_n$ is from a neighborhood of infinity in $\mathbb{H}_n^+.$ 
\end{lemma}
\begin{proof}
Recall that $R_{\mu,n}(b_n)=G^{\langle-1\rangle}_{\mu,n}(b_n)- b_n^{-1}$ for  $b_n\in \mathbb{H}^-_n$ such that $\|b_n\|$ is small enough. Let $b_n\in\mathbb{H}^+_n$, then $G_{\mu\br\nu}(b_n)\in \mathbb{H}_n^-$ and we have that 
\begin{align*}
R_{\mu\br\nu,n}(G_{\mu\br \nu,n}(b_n))&=b_n-(G_{\mu\br \nu,n}(b_n))^{-1}\\
&=b_n-F_{\mu\br \nu,n}(b_n).
\end{align*}

By Proposition \ref{reciprocal cauchy of free-Boolean},  if $\|b_n^{-1}\|$ is small enough, we have 
\begin{align*}
R_{\mu\br\nu,n}(G_{\mu\br \nu,n}(b_n))&=b_n-F_{\mu\br \nu,n}(b_n)\\
&=F_{\nu\br \mu,n}(b_n)-F_{\nu\bp \mu,n}(b_n)\\
&=F^{\langle -1\rangle}_{\mu,n}[F_{\nu\bp\mu,n}(b_n)]-F_{\nu\bp \mu,n}(b_n).
\end{align*}

Notice that $F^{\langle -1\rangle}_{\mu,n}[F_{\nu\br \mu,n}(b_n)]-F_{\nu\br \mu,n}(b_n)=\phi_{\mu,n}(F_{\nu\br \mu,n}(b_n))$. where $\phi_{\mu,n}$ is the Voiculescu transform  and $\phi_{\mu,n}(b)=R_{\mu,n}(b^{-1}) $.

 It follows that 
\begin{equation}
R_{\mu\br\nu,n}(G_{\mu\br \nu,n}(b_n))=R_{\mu,n}(G_{\mu\bp \nu,n}(b_n)).
\end{equation}

By Proposition \ref{reciprocal cauchy of Free-Monotone}, we have that 
\begin{align*}
G_{\mu\bp \nu,n}(b_n)&=[F_{\mu\bp \nu,n}(b_n)]^{-1}\\&
= [F_{\nu,n}(F_{\mu\br \nu,n}(b_n))]^{-1}\\
=G_{\nu,n}(F_{\mu\br \nu,n}(b_n)).
\end{align*}

Therefore, we have that 
$$R_{\mu\br\nu,n}(G_{\mu\br \nu,n}(b_n))=R_{\mu,n}[G_{\nu,n}(F_{\mu\br \nu,n}(b_n)).]$$

Since the range of $G_{\mu\br \nu,n}$ contains a neighborhood of $0$ in $\mathbb{H}^-_n.$
The proof is done.
\end{proof}

By the linearization property of the R-transform, we have the following result.
\begin{proposition}\label{R-transform and subordination}\normalfont
Let $\mu_1,\mu_2,\mu,\nu\in \cbd$ and $\alpha$ is a completely positive map from $\range$ to $\range$. Then
$$(\mu_1\bp\mu_2)\br\nu=(\mu_1\br\nu)\bp(\mu_2\br\nu).$$ 
If $\mu^{\bp \alpha}\in\cbd$, then 
$$(\mu\br\nu)^{\bp \alpha}\in\cbd$$ and 
$$(\mu\br\nu)^{\bp \alpha}= \mu^{\bp \alpha}\br\nu.$$
\end{proposition}

\section{More relations between convolutions and transforms in operator-valued free probability}
In this section, we use the properties in previous sections to study more relations between convolutions and transforms in operator-valued free probability. The main fact we use in the section is that the matricial reciprocal Cauchy transforms,  the matricial R transforms, Voiculescu's transform can completely determine $\range$-valued distributions. We will denote by $\cp$ the family of completely positive maps from $\range$ to $\range$. Given $s\in\cp$, for each $n\in \N$, $s_n$ denotes the completely map from $\range\otimes\mnc$  to $\range\otimes\mnc$ defined as 
$$s_n(b\otimes e(i,j))=s(b)\otimes e,$$
where $b\in\range$ and $e\in\mnc$.

\begin{proposition} Let  $\mu,\nu\in\cbd$. Then
$$\mu\vdash(\nu\br\mu)=\mu\br\nu.$$
\end{proposition}
\begin{proof}
For each $n\in \N$, $b_n\in\mathbb{H}_n^+$, by Proposition \ref{reciprocal cauchy of orthogonal}, \ref{reciprocal cauchy of Free-Monotone},\ref{reciprocal cauchy of free-Boolean}, we have 
\begin{align*}
F_{\mu\vdash(\nu\br\mu),n}(b_n)&=F_{\mu,n}(F_{\nu\br\mu,n}(b_n))-F_{\nu\br\mu,n}(b_n)+b_n\\
&=F_{\nu\bp\mu,n}(b_n)-F_{\nu\br\mu,n}(b_n)+b_n\\
&=F_{\mu\br\nu,n}(b_n).\\
\end{align*}
The proof is done.
\end{proof}

Recall that a $\range$-valued Boolean central limit law or say $\range$-valued Bernoulli law is completely determine by its variance which is a completely positive map $s\in\cp$ \cite{BPV}. If we denote by $Ber_s$ be the $\range$-valued Bernoulli law of variance $s$, then
$$F_{Ber_s,n}(b_n)=b_n-s_n(b_n^{-1}),$$
for $b_n\in\mathbb{H}_n^+.$ We will write $Ber$ short for $Ber_1$.

On the other hand , a $\range$-valued free central limit law or say $\range$-valued semicircularlaw is completely determine by its variance which is a completely positive map $s\in\cp$ \cite{BPV}. If we denote by $\gamma_s$ the $\range$-valued semicircular law of variance $s$, then the matricial Voiculescu's transform $(\phi_{\gamma,n})$ are given by the following formula
$$\phi_{\gamma,n}(b_n)=s_n(b_n^{-1}),$$
for $b_n\in\mathbb{H}_n^+$ in a neighborhood of infinity. We will write $\gamma$ short for $\gamma_1$.

\begin{proposition}\normalfont Given $s\in\cp$, we have 
$$Ber_{s}\vdash \gamma_s=\gamma_s.$$
\end{proposition}
\begin{proof}
By Proposition \ref{reciprocal cauchy of orthogonal}, for $n\in \N$, we have that 
\begin{align*}
F_{Ber_{s}\vdash \gamma_s,n}(b_n)&=F_{Ber_{s},n}(F_{\gamma_s,n}(b_n))-F_{\gamma_s,n}(b_n)+b_n\\
&=F_{\gamma_s,n}(b_n)-  s_n(G_{\gamma_s,n}(b_n)) -F_{\gamma_s,n}(b_n)+b_n\\
&=F_{\gamma_s,n}(b_n)- \phi_{\gamma_s,n}(F_{\gamma_s,n}(b_n)))-F_{\gamma_s,n}(b_n)+b_n.
\end{align*}

 According to the definition of the matricial Voiculescu transform, we have that 
 $$ \phi_{\mu,n}(b_n)=F_{\mu,n}^{\langle-1\rangle}(b_n)-b_n.$$

Replace  $b_n$ by$F_{\mu,n}(b_n)$, then we have   we have 
$$ \phi_{\mu,n}(F_{\mu,n}(b_n))=b_n-F_{\mu,n}(b_n).$$

Therefore, we have 
$$F_{Ber_{s}\vdash \gamma_s,n}(b_n)=F_{\gamma_s,n}(b_n).$$
\end{proof}

An operator valued Belinschi-Nica transform $\Phi$ is defined combinatorially  in \cite{ABFN}. It is not clear that if the transform is a map from $\cbd$ to $\cbd$.  Below, we give an equivalent definition of $\Phi$ by a map from $\cbd$ to $\cbd$ which is an operator-valued analogue of relation (\ref{Phi-transform}).
\begin{definition}\normalfont Let $\mu\in\cbd.$  We define $\Phi:\cbd\rightarrow \cbd$ such that 
$$\Phi(\mu)=Ber\vdash \mu,$$
for $\mu\in \cbd$. Then, for each $n\in\N$, we have 
that
$$F_{\Phi(\mu),n}(b_n)=b_n-G_{\mu,n}(b_n),$$
for $b_n\in \mathbb{H}_n^+.$
\end{definition}

The following proposition exhibits a relation between the standard semicircular law, $\Phi$-transform and the s-free additive convolution.
\begin{proposition}\label{Phi suboridination semicirlucar}\normalfont Let $\mu\in\cbd$.  Then we have
$$\gamma\br\mu=\Phi(\gamma \bp \mu). $$
\end{proposition}
\begin{proof}For each $n\in\N$, $b_n\in \mathbb{H}_n^+$, we have
\begin{align*}
    F_{\gamma\br\mu,n}(b_n)&=F_{\gamma \bp \mu,n}(b_n)-F_{\mu\br \gamma,n}(b_n) +b_n\\
    &=F_{\gamma \bp \mu,n}(b_n)+G_{\gamma \bp \mu,}(b_n)-F_{\mu\br \gamma,}(b_n)-G_{\gamma \bp \mu,}(b_n)+b_n\\
    &=F^{<-1>}_{\gamma,n}( F_{\gamma \bp \mu,n}(b_n)) -F_{\mu\br \gamma,n}(b_n)-G_{\gamma \bp \mu,n}(b_n)+b_n\\
    &=-G_{\gamma \bp \mu,n}(b_n)+b_n\\
    &=F_{\Phi(\gamma \bp \mu),n}(b_n).
\end{align*}
The proof is done.
\end{proof}

Now we turn to study the $\range$-valued $\B$-transform which is already given in \cite{ABFN}.
\begin{definition}\normalfont
Let $s\in\cp$. The map $\B_s:\cbd\rightarrow \cbd$ is defined by the following formula
$$ \B_s(\mu)=(\mu^{\bp(1+s)})^{\uplus(1+s)^{-1}},$$
for $\mu\in\cbd$.
\end{definition}

The following proposition is proven combinatorially  in \cite{ABFN}.
\begin{proposition}\label{B-semigroup}\normalfont
Given $s,t\in\cp$, then we have
$$\B_s(\B_t(\mu))=\B_{s+t}(\mu),$$
for all $\mu\in\cbd. $
\end{proposition}

One would see that this proposition can be derived from the following result which is a generalization of Relation
(\ref{scalar FB-BF}).

\begin{proposition}\label{exchange free and boolean}\normalfont
Let $p,q\in\cp$ such that $p-1$  and $1-p+qp$ are invertible completely positive maps.  Then we have that 
$$(\mu^{\bp p})^{\uplus q}=(\mu^{\uplus q'})^{\bp p'}, $$
where $p',q'\in\cp$ are  defined as follows:
$$ p'=qp(1-p+qp)^{-1},\quad q'=1-p+qp.$$
\end{proposition} 
\begin{proof}We just need to compare the matricial reciprocal Cauchy transforms of the two distributions.
By Proposition \ref{reciprocal Cauchy transform of Boolean operator power convolution}, for each $n\in\N$ and $b_n\in \mathbb{H}_n^+$, we have 
$$F_{(\mu^{\bp p})^{\uplus q},n}(b_n)=q_n[F_{\mu^{\bp p},n}(b_n)]+(1_n-q_n)b_n.$$
By Proposition \ref{reciprocal Cauchy transform of  s-free convolution powers}, and equation (\ref{operator power subordination 2}), we have that  
$$ F_{\mu^{\bp p},n}(b_n)=(p_n-1_n)^{-1}p_nF_{\mu^{\uplus (p-1)}\br \mu^{\uplus (p-1)},n}(b_n)-(p_n-1_n)^{-1}b_n.$$
Therefore, we have that 
$$F_{(\mu^{\bp p})^{\uplus q},n}(b_n)
=q_n(p_n-1+n)^{-1}p_nF_{\mu^{\uplus (p-1)}\br \mu^{\uplus (p-1)},n}(b_n)+(1_n-q_n(p_n-1)^{-1}p_n)b_n.$$

On the other hand, apply Proposition \ref{reciprocal Cauchy transform of  s-free convolution powers} to $F_{(\mu^{\uplus q'})^{\bp p'},n}(b_n)$ we have that 
\begin{align*}
F_{(\mu^{\uplus q'})^{\bp p'}}(b_n)&=(p_n'-1_n)^{-1}t\omega_{q',p'}(b_n)+(1_n-(p_n'-1)^{-1}q'_n)b_n,
\end{align*}
where $\omega_{q',p',n}(b_n)=F_{(\mu^{\uplus q'})^{\uplus (p'-1)}\br(\mu^{\uplus q'})^{\uplus (p'-1)},n}(b_n)=F_{\mu^{\uplus q'(p'-1)}\br \mu^{\uplus q'(p'-1)},n}(b_n)$.

Let $q'=1-p+qp$ and $p'=qp(1-p+qp)^{-1}$. Then we have 
$$ q'(p'-1)=p-1$$
and 
$$ (p'-1)^{-1}t=q(p-1)^{-1}. $$
In this case, we have that 
$$F_{(\mu^{\uplus q'})^{\bp p'}}(b_n)=F_{(\mu^{\uplus q'})^{\bp p'}}(b_n).$$
The proof is done.
\end{proof}

The following is a proof of Proposition \ref{B-semigroup} by use the Proposition \ref{exchange free and boolean}.
\begin{proof}
Apply Proposition \ref{exchange free and boolean} by setting $p=(s+1)(s+t+1)^{-1} $ and $q=(s+t+1)(t+1)^{-1}$, 
then  $p'=s+1$ and $q'=(1+t)^{-1}$.   Therefore, we have 
$$(\mu^{\bp (s+1)(s+t+1)^{-1}})^{\uplus (s+t+1)(t+1)^{-1}}=(\mu^{\uplus (1+t)^{-1}})^{\bp s+1}.$$
It follows that 
\begin{align*}
\B_s(\B_t(\mu))&= (((\mu^{\bp 1+t})^{\uplus (1+t)^{-1}})^{\bp 1+s})^{\uplus (1+s)^{-1}}\\
&= (((\mu^{\bp 1+t})^{\bp (1+t+s)(1+t)^{-1}})^{\uplus (1+s)(1+t+s)^{-1}})^{\uplus (1+s)^{-1}}\\
&= (\mu^{\bp 1+t+s})^{\uplus (1+t+s)^{-1}}.
\end{align*}
The proof is done.
\end{proof}

We provide one more application of Proposition \ref{exchange free and boolean} in studying the $\B$-transforms.  Given  $s\in\cp$ such that  $s\geq 1$ and $\mu\in\cbd$, we have that $\mu^{\bp s}\in\cbd$ \cite{ABFN, Sh1}. By Theorem 8.4 in \cite{ABFN} and Proposition \ref{reciprocal Cauchy transform of Boolean operator power convolution}, we have that $\mu^{\uplus(s-1)}\br \mu^{\uplus(s-1)}$ is the distribution subordinate to $\mu^{\br s}$ with respect to $\mu$, namely 
$$F_{\mu,n}(F_{\mu^{\uplus(s-1)}\br \mu^{\uplus(s-1)}})=F_{\mu^{\bp s},n},$$
for all $n$. The following result is an operator valued generalization of Proposition 1.12 in \cite{Nica}.
\begin{proposition}\normalfont\label{infinitely divisibility}
Given $s\in\cp$ and $\mu\cbd$, then we have 
$$\mu^{\uplus s}\br \mu^{\uplus s}=(\B(\mu))^{\bp s}.$$
\end{proposition}
\begin{proof}
By Proposition \ref{subordination, Boolean}, we have that 
\begin{align*}
\mu^{\uplus s}\bp \mu^{\uplus s}&=[\mu^{\uplus s}\br \mu^{\uplus s}]\uplus[\mu^{\uplus s}\br \mu^{\uplus s}]\\
&=[\mu^{\uplus s}\br \mu^{\uplus s}]^{\uplus 2}.
\end{align*}

Therefore, we have 
$$ \mu^{\uplus s}\br \mu^{\uplus s}=[\mu^{\uplus s}\bp \mu^{\uplus s}]^{\uplus 1/2}.$$
Let $p=1+s$ and $q=2s(1+s)^{-1}$.  Then the $p',q'$ in Proposition \ref{exchange free and boolean} are
$$p'=2,\quad q'=s.$$ 

It follows that 
\begin{align*}
((\mu^{\uplus s})^{\bp 2})^{\uplus 1/2}&=((\mu^{\bp 1+s})^{\bp 2s(1+s)})^{\uplus 1/2}\\
&=(\mu^{\bp 1+s})^{\bp s(1+s)}.
\end{align*}

On the other hand, let  $p=(1+s)/2$ and $q=s(1+s)^{-1}$.  Then the $p',q'$ in Proposition \ref{exchange free and boolean} are
$$p'=s,\quad q'=1/2.$$ 
\begin{align*}
 (\B(\mu))^{\bp s}&=((\mu^{\bp 2})^{\uplus 1/2})^{\uplus s}\\
&=((\mu^{\bp 2})^{\bp (1+s)/2})^{\uplus s(1+s)^{-1}}\\
&=(\mu^{\bp 1+s})^{\bp s(1+s)}.\\
\end{align*}
The proof is done.
\end{proof}

\begin{definition} \normalfont Let $\mu\in\cbd$.  $\mu$ is $\bp$-infinitely divisible if $\mu^{\bp 1/n}\in \cbd$ for all $n\geq 1.$
\end{definition}
We see that in Proposition \ref{infinitely divisibility},  we actually proved the the following relation
$$ \B(\mu^{\uplus s})=\B(\mu)^{\bp s}.$$
Let $s\in\cp$ be invertible. Then we have 
\begin{equation}
\B(\mu)=[\B(\mu^{\uplus s^{-1}})]^{\bp s}.
\end{equation}
Notice that $\mu^{\uplus s^{-1}}\in\cbd$ if $\nu\in\cbd.$  Therefore,  $\B(\mu)$ is $\bp$-infinitely divisible for all $\mu\in\cbd$ which is the property of the scalar Bercovici-Pata bijection. 

Now, we turn to study  an operator valued generalization of Anshelevich's free convolution with two states.
\begin{lemma}\label{Voiculescu transform and inverse reciprocal Cauchy transform}\normalfont
  Let $\mu\in\cbd$ and  $s$ is a completely positive map from $\range$ to $\range$ such $\mu^{\bp s}\in \cbd$. Then, for each $n]\in \N$, we have  that 
  $$ F_{\mu^{\bp s},n}^{\langle-1\rangle}(b_n)=s_n[F_{\mu,n}^{\langle-1\rangle}(b_n)]+(1_n-s_n)b_n,$$
  $b_n\in\mathbb{H}^+_n$ is a in sufficiently small neighborhood of infinity.
\end{lemma}
\begin{proof}
For each $n\in \N$,  we have that  
$$ \phi_{\mu,n}(b_n)=F_{\mu,n}^{\langle-1\rangle}(b_n)-b_n.$$
and 
$$ \phi_{\mu^{\bp s},n}(b_n)=F_{\mu^{\bp s},n}^{\langle-1\rangle}(b_n)-b_n.$$
Notice that $$ \phi_{\mu^{\bp s},n}(b_n)=s_n[\phi_{\mu,n}(b_n)].$$
The result follows by simplifying the above equation.
\end{proof}

\begin{proposition}\label{operator free envolution}\normalfont
Given $\mu,\nu\in\cbd$ and $s\in\cp$  such that $\mu^{\bp s}\in \cbd$, then we have 
$$ \B_{s}(\mu\br \nu)=\mu\br(\mu^{\bp s}\bp\nu).$$
\end{proposition}
\begin{proof}
For each $n\in \N$, $b_n\in \mathbb{H}^+_n$, by Proposition\ref{reciprocal cauchy of free-Boolean}, we have 
\begin{equation}\label{proof of operator free envolution}
\begin{aligned}
F_{\mu\br(\mu^{\bp s}\bp\nu),n}(b_n)&= F_{\mu\bp(\mu^{\bp s}\bp\nu),n}(b_n)-F_{(\mu^{\bp s}\bp\nu)\br,n}(b_n)+b_n\\
&= F_{\mu^{\bp (s+1)}\bp\nu,n}(b_n)-F_{(\mu^{\bp s}\bp\nu)\br\mu,n}(b_n)+b_n\\
&= F_{\mu^{\bp (s+1)}\br\nu,n}(b_n)+F_{\nu\br\mu^{\bp (s+1)},n}(b_n)-F_{(\mu^{\bp s}\bp\nu)\br\mu,n}(b_n)\\
&= F_{\mu^{\bp (s+1)}\br\nu,n}(b_n)+F_{\nu\br\mu^{\bp (s+1)},n}(b_n)-F_{\mu,n}^{\langle-1\rangle}[F_{\mu^{\bp (s+1)}\bp\nu,n}(b_n)].\\
\end{aligned}
\end{equation}

Notice that $1+s$ is invertible, by Lemma \ref{Voiculescu transform and inverse reciprocal Cauchy transform}, we have that 
$$ F_{\mu^{\bp (1+s)},n}^{\langle-1\rangle}(b_n)=(1+s_n)[F_{\mu,n}^{\langle-1\rangle}(b_n)]-s_nb_n.$$

Therefore, we have 
$$ (1+s_n)^{-1}[ F_{\mu^{\bp (1+s)},n}^{\langle-1\rangle}(b_n)+s_nb_n]=F_{\mu,n}^{\langle-1\rangle}(b_n).$$

It follows that 
\begin{align*}
&F_{\mu,n}^{\langle-1\rangle}[F_{\mu^{\bp (s+1)}\bp\nu,n}(b_n)]\\
=&(1+s_n)^{-1}[ F_{\mu^{\bp (1+s)},n}^{\langle-1\rangle}(F_{\mu^{\bp (s+1)}\bp\nu,n}(b_n))+s_nF_{\mu^{\bp (s+1)}\bp\nu,n}(b_n)]\\
=&(1+s_n)^{-1}[ F_{\mu^{\bp (1+s)},n}^{\langle-1\rangle}(F_{\mu^{\bp (s+1)}\bp\nu,n}(b_n))]+(1+s_n)^{-1}s_n[F_{\mu^{\bp (s+1)}\bp\nu,n}(b_n)]\\
=&(1+s_n)^{-1}[ F_{\nu\br\mu^{\bp (s+1)},n}(b_n)]+(1+s_n)^{-1}s_n[F_{\mu^{\bp (s+1)}\bp\nu,n}(b_n)].\\
\end{align*}

Then, the Equation (\ref{proof of operator free envolution}) becomes 
\begin{align*}
&F_{\mu\br(\mu^{\bp s}\bp\nu),n}(b_n)\\
=&F_{\mu^{\bp (s+1)}\br\nu,n}(b_n)+F_{\nu\br\mu^{\bp (s+1)},n}(b_n)-\{(1+s_n)^{-1}[ F_{\nu\br\mu^{\bp (s+1)},n}(b_n)]+(1+s_n)^{-1}s_n[F_{\mu^{\bp (s+1)}\bp\nu,n}(b_n)]\}\\
=&F_{\mu^{\bp (s+1)}\br\nu,n}(b_n)+(1+s_n)^{-1}s_n[ F_{\nu\br\mu^{\bp (s+1)},n}(b_n)]-(1+s_n)^{-1}s_n[F_{\mu^{\bp (s+1)}\bp\nu,n}(b_n)]\\
=&F_{\mu^{\bp (s+1)}\br\nu,n}(b_n)+(1+s_n)^{-1}s_n[ F_{\nu\br\mu^{\bp (s+1)},n}(b_n)-F_{\mu^{\bp (s+1)}\bp\nu,n}(b_n)]\\
=&F_{\mu^{\bp (s+1)}\br\nu,n}(b_n)+(1+s_n)^{-1}s_n[ -F_{\mu^{\bp (s+1)}\br\nu,n}(b_n)+b_n]\\
=&(1+s_n)^{-1}[F_{\mu^{\bp (s+1)}\br\nu,n}(b_n)]+(1-(1+s_n)^{-1})b_n.
\end{align*}

By Proposition \ref{reciprocal Cauchy transform of Boolean operator power convolution},  we have 
$$F_{\mu\br(\mu^{\bp s}\bp\nu),n}(b_n)=F_{(\mu^{\bp (s+1)}\br\nu)^{\uplus (1+s)^{-1}},n}(b_n).$$

Therefore, 
  $$\mu\br(\mu^{\bp s}\bp\nu)=(\mu^{\bp (s+1)}\br\nu)^{\uplus (1+s)^{-1}}.$$

By Proposition \ref{R-transform and subordination}, we have 
$$(\mu^{\bp (s+1)}\br\nu)^{\uplus (1+s)^{-1}}=((\mu\br\nu)^{\bp (s+1)})^{\uplus (1+s)^{-1}}=\B_s(\mu\br\nu).$$
The proof is done.
\end{proof}

\begin{corollary}\normalfont
Let $\nu\in\cbd$,  $s\in\cp$  and $\gamma_s$ be the semicircular of variance $s$. Then
 $$ \Phi(\gamma_s\bp \nu)=\B_s(\Phi(\nu)).$$
\end{corollary}
\begin{proof}
The statement follows Proposition \ref{Phi suboridination semicirlucar} and Proposition \ref{operator free envolution} by letting $\mu=\gamma$.
\end{proof}

\bibliographystyle{plain}

\bibliography{references}

\noindent Department of Mathematics\\
Indiana University	\\
Bloomington, IN 47401, USA\\
E-MAIL: liuweih@indiana.edu \\

\end{document}